\documentclass[11pt,oneside]{amsart}
\usepackage{amsmath,amsthm,amssymb,amscd,fancybox,esint}
\usepackage[a4paper,left=3cm,right=3cm,top=3cm,bottom=3cm]{geometry}
\usepackage{cite}

\allowdisplaybreaks
\renewcommand{\epsilon}{\varepsilon}
\renewcommand{\phi}{\varphi}

 \newcommand{\bZ}{\mathbb{Z}}
\newcommand{\bR}{\mathbb{R}} \newcommand{\bN}{\mathbb{N}}
\newcommand{\cF}{\mathcal{F}} 
\newcommand{\dS}{\mathcal{S}(\bR^d)} 

\newcommand{\ddS}{\mathcal{S}'(\bR^d)}

\newcommand{\supp}{\operatorname{supp}}

\newcommand{\Zn}{\mathbb{Z}^d}

\newcommand{\Rd}{\mathbb{R}^d}

\newcommand{\brac}[1]{\langle #1\rangle}
\newtheorem{theorem}{Theorem}[section]
\newtheorem{lemma}[theorem]{Lemma}
\newtheorem{proposition}[theorem]{Proposition}

\newtheorem{corollary}[theorem]{Corollary}
\theoremstyle{definition}
\newtheorem{definition}[theorem]{Definition}
\newtheorem{example}[theorem]{Example}
\theoremstyle{remark}
\newtheorem{remark}[theorem]{Remark}
\numberwithin{equation}{section}
\renewcommand{\d}{\,\text{d}}
\def\dx{{\,\operatorname{d}\!x}}

\def\dt{{\,\operatorname{d}\!t}}

\def\bZ{{\mathbb Z}}

\def\bN{{\mathbb N}}

\def\bC{{\mathbb C}}
\def\bR{{\mathbb R}}

\def\cB{\mathcal{B}}
\def\cC{\mathcal{C}}
\def\cD{{\mathbb{R}^d}}

\def\cF{\mathcal{F}}
\def\Z{\mathbb{Z}}

\def\cP{\mathcal{P}}
\def\cQ{\mathcal{Q}}
\def\cR{\mathcal{R}}
\def\cS{\mathcal{S}}

\newcommand{\C}{\mathbb{C}}
\newcommand{\R}{\mathbb{R}}
\newcommand{\vf}{\mathbf{f}}
\newcommand{\vg}{\mathbf{g}}

\newcommand{\vx}{\mathbf{x}}
\newcommand{\vc}{\mathbf{c}}
\newcommand{\vd}{\mathbf{d}}

\newcommand{\vs}{\mathbf{s}}

\newcommand{\vy}{\mathbf{y}}
\newcommand{\vz}{\mathbf{z}}
\newcommand{\vt}{\mathbf{t}}
\providecommand{\abs}[2][]{#1\lvert#2#1\rvert}
\providecommand{\norm}[2][]{#1\lVert#2#1\rVert}
\newcommand{\esssup}{\mathop{\rm ess{\,}sup}}
\def\cprime{$'$} 
\begin{document}
\title[Matrix-weighted Anisotropic Smoothness Spaces]{ Matrix-weighted Anisotropic Smoothness Spaces}
\author[M.\ Nielsen]{Morten Nielsen}
\thanks{This work was supported by the Independent Research Fund Denmark, grant no.\ 5281-00046B}
\address{Department of Mathematical Sciences\\ Aalborg
  University\\ Thomas Manns Vej 23\\ DK-9220 Aalborg East\\ Denmark}
\email{mnielsen@math.aau.dk}
\begin{abstract}
Given a quasi-norm on $\bR^d$ induced by a one-parameter dilation group, we consider matrix weights $W$ in an adapted Muckenhoupt class $\mathbf{A}_p$, $0 < p < \infty$, and use these weights to introduce and study anisotropic matrix-weighted smoothness spaces in both continuous and discrete settings. The spaces are constructed by means of a decomposition method in the frequency domain. We prove the equivalence of the continuous and discrete spaces using suitably adapted tight frames. Compatible notions of molecules and almost diagonal matrices are also introduced, and applications to the study of Fourier multipliers and pseudo-differential operators on vector-valued smoothness spaces are given.
\end{abstract}
\maketitle

\section{Introduction}
Given a measurable matrix-valued function $W\colon
\cD\to\C^{N\times N}$ that is positive definite almost everywhere, 
the matrix-weighted $L^p$-space $L^p(W)$, $0< p<\infty$, is defined
 as the family of
measurable functions $\vf\colon \cD\to\C^N$ satisfying  
\begin{equation}\label{eq:lp}
  \norm{\vf}_{L^p(W)}:=\Biggl(\int_{\cD}\abs{W^{1/p}(x)\vf(x)}^p\,dx\Biggr)^{1/p}<\infty.
\end{equation}
Taking the quotient by 
 $$\{\vf\colon \cD\to\C^N;\norm{\vf}_{L^p(W)}=0\},$$
turns $L^p(W)$ into a (quasi-)Banach space. 

The study of linear operators on matrix-weighted $L^p$-spaces has attracted a great deal of attention recently, see, e.g., \cite{IsraPottTrei:2022a,NielSiki:2021a,Roud:2003a,Niel:2010a}, due to the fact that the setting provides a natural extension of scalar-weighted $L^p$-spaces to a vector-valued setting and the extended setting opens a number of interesting mathematical questions related to classical results on Muckenhoupt weights in harmonic analysis. A highlight in the theory of matrix weights is the formulation of a suitable matrix Muckenhoupt $A_p$ condition by Treil and Volberg that completely characterizes boundedness of the Hilbert transform on $L^p(W)$ for weights on $\mathbb{R}$, $1 < p < \infty$, see \cite{TreiVolb:1997a}, with extensions to general $p$ given by Nazarov and Treil \cite{NazaTrei:1996a} and, via a different approach, by Volberg \cite{Volb:1997a}.

 As is well known in the scalar case, one can use $L^p$-spaces to build a variety of useful smoothness spaces by imposing restrictions on suitable local components of functions measured by a (weighted) $L^p$-norm. Roudenko was the first to apply such an approach in the matrix-weighted setup, based on $L^p(W)$ spaces as defined in Eq.\ \eqref{eq:lp}, see \cite{Roud:2003a}, where she introduced a very natural notion of matrix-weighted Besov spaces $B^s_{p,q}(W)$ provided that the weight $W$ satisfies a Muckenhoupt $A_p$ condition. 
 This work was later extended by Frazier and Roudenko, first to the full range $0<p\le 1$ for Besov spaces \cite{FrazRoud:2004a}, and subsequently to matrix-weighted Triebel-Lizorkin spaces \cite{ FrazRoud:2021a}.
  An extensive recent study by Bu et al.\ of various properties of matrix-weighted Besov and Triebel-Lizorkin type spaces, including almost diagonal operators and molecular decompositions, can be found in \cite{BuHytoYang:2024a,BuHytoYang:2025b,BuHytoYang:2025c}, formulated for the classical dyadic setting.

   Besov spaces are created by measuring $L^p$-norms of local components of functions corresponding to a dyadic decomposition of the frequency space. However, it was observed in the scalar case by Triebel \cite{Trie:1983a} that the same general decomposition approach, using other partitions of the frequency space, can yield other types of useful smoothness spaces such as modulation spaces that are associated with a uniform decomposition of the frequency space. The decomposition approach was formalized and studied further by Feichtinger and Gr\"obner \cite{FeicGrob:1985a, Feic:1987a}.  
Another benefit from the decomposition approach is that it provides a natural framework to study anisotropic function spaces, and the anisotropic setup can often be treated with only a minimal increase in technical complexity, see, e.g.,\ \cite{BoruNiel:2008a,CleaGeor:2020a}. 
 
The main contribution of the present paper is to present a construction of anisotropic matrix-weighted smoothness spaces of vector-valued functions defined on $\mathbb{R}^d$, extending the scalar-weighted anisotropic decomposition spaces of \cite{BoruNiel:2008a} to the matrix-weighted setting.
 The spaces correspond to flexible but
structured decompositions of the frequency space $\bR^d$.  The
decompositions are generated by a one-parameter group of dilations on
$\bR^d$ and a suitable decomposition function. The group of dilations
incorporates the anisotropy of the construction, while the
decomposition function governs the general splitting of the frequency
space.  
 There is, however, one significant technical obstacle to such a construction in the matrix-weighted case. The ``standard'' Muckenhoupt $A_p$ condition for matrix weights turns out not to be compatible with the mentioned general anisotropic structure. A solution to this incompatibility was proposed by the author in \cite{mn_geo}, where modified Muckenhoupt conditions are introduced to create the needed compatibility with the anisotropic setup. We shall rely on the same approach in the present article to incorporate the anisotropic structure.

The structure of the paper is as follows. In Section \ref{sec:HS}, we first recall some needed facts about quasi-norms on $\bR^d$ induced by a one-parameter dilation group, and we define Muckenhoupt conditions for matrix weights adapted to the quasi-norm setting following \cite{{mn_geo}}. In Section \ref{sec:3} we 
introduce the main class of matrix-weighted smoothness spaces. Questions related to completeness and stability of the smoothness spaces are also addressed in Section~\ref{sec:3}.  

In Section \ref{sec:ms} we introduce a suitable family of weighted coefficient spaces of vector-valued sequences together with a simple construction of tight frames for the matrix-weighted smoothness spaces. A complete discrete characterization of the matrix-weighted smoothness spaces is then obtained using the frames. An algebra of discrete almost diagonal matrices adapted to the matrix-weighted decomposition spaces is introduced in Section \ref{sec:4}, which allows us to define a natural notion of molecules for matrix-weighted decomposition spaces. The almost diagonal matrices and molecules may be used to simplify the study of various operators on the matrix-weighted smoothness spaces, making it much easier to obtain various boundedness results for, e.g., partial differential operators. As an example of the almost diagonal approach, we conclude the paper in Section \ref{sec:5} with a study of Fourier multipliers on matrix-weighted decomposition spaces.

\section{Anisotropic quasi-norms on $\bR^d$ and Compatible Muckenhoupt Weights}\label{sec:HS}
Let us now define spaces of homogeneous type on $\cD$ with a suitable anisotropic structure that will later serve as the framework for building smoothness spaces using suitable coverings of the frequency space. The anisotropic structure on $\cD$ is created using a quasi-norm induced by a one-parameter group of dilations.         

We let $|\cdot|$ denote the Euclidean norm on $\bR^d$ induced by the corresponding inner product
$\langle \cdot,\cdot\rangle$ and let $A$ be a given real symmetric $d\times d$ matrix with strictly positive eigenvalues. For $t>0$, define the group of dilations $\delta_t:\bR^d\to\bR^d$
  by $\delta_t := \exp (A\ln t)$ and let $\nu:=\textrm{trace}(A)$. Dilation groups with this structure have been studied in, e.g., \cite[Part II]{SteiWain:1978a}.  
  
According to \cite[Proposition 1.7]{SteiWain:1978a}, there exists a strictly positive
symmetric matrix $P$ such that for all $\xi \in \bR^d$,
\begin{equation}\label{eq:P}
[\delta_t\xi]_P:=\langle P\delta_t\xi,\delta_t\xi \rangle^{\tfrac{1}{2}}
\end{equation}
is a strictly increasing function of $t$. This helps us introduce a quasi-norm $|\cdot|_A$ associated with $A$.
\begin{definition}\label{def:A}
We define the function $|\cdot|_A:\bR^d\to \bR_+$ by $|0|_A:=0$ and for
$\xi\in\bR^d\backslash\{0\}$ by letting $|\xi|_A$ be the unique solution $t$ to the
equation $[\delta_{1/ t}\xi]_P=1$.
\end{definition}

For notational convenience, let $\brac{\cdot}_A:=1+|\cdot|_A$ denote the \textit{bracket function}. It can be verified (see, e.g., \cite{SteiWain:1978a, BoruNiel:2008a}) that:
\begin{itemize}
\item There exists a constant $C_A\geq 1$ such that
\begin{equation}\label{eq:Qtriangle}
|\xi+\zeta|_A\le C_A(|\xi|_A+|\zeta|_A), \, \xi,\zeta \in \bR^d.
\end{equation}
\item We have the homogeneity property: $|\delta_t\xi|_A=t|\xi|_A$, $t>0$.
\item There exist constants $C_1,C_2,a_-,a_+>0$ such that
\begin{equation}\label{natasha}
C_1\min(|\xi|^{a_-}_A,|\xi|^{a_+}_A) \le |\xi|\le
C_2\max(|\xi|^{a_-}_A,|\xi|^{a_+}_A), \qquad \xi \in \bR^d.
\end{equation}
\item For any $\epsilon>0$, there exists $C_\epsilon<\infty$ such that
\begin{equation}\label{eq:integral}
\int_{\bR^d}\brac{x}_A^{-\nu-\epsilon}\,\d x\leq C_\epsilon.
\end{equation}

\end{itemize}
In fact, let $\sigma(A)$ denote the spectrum of $A$. Then we may select
$a_-$ and $a_+$ in \eqref{natasha} as follows:
\begin{equation}
  \label{eq:spec}
  a_-=\min_{\lambda\in\sigma(A)} \lambda\leq a_+=\max_{\lambda\in\sigma(A)} \lambda.
\end{equation}

\begin{example}
For $A=\textrm{diag}(\beta_1,\beta_2,\ldots,\beta_d)$, $\beta_i>0$, we have
$\delta_t=\textrm{diag}(t^{\beta_1},t^{\beta_2},\ldots,t^{\beta_d})$, and one may verify
that
\begin{equation*}
|\xi|_A\asymp \sum_{j=1}^d|\xi_j|^{\tfrac{1}{\beta_j}}, \qquad \xi \in\bR^d.
\end{equation*}
In particular, for $A=\textrm{diag}(1,1,\ldots,1)$, we recover (up to equivalence) the usual isotropic Euclidean norm on $\bR^d$.
\end{example}

We define the balls ${B}_A(\xi,r):=\{\zeta\in\bR^d:|\xi-\zeta|_A<r\}$. It
can be verified that $|{B}_A(\xi,r)|=r^{\nu}\omega_d^A$, where
$\omega_d^A:=|{B}_A(0,1)|$, and consequently $(\bR^d,|\cdot|_A,d\xi)$ is a space of homogeneous
type with homogeneous dimension $\nu$.

We let 
\begin{equation}\label{eq:ball}\cB_A:=\{{B}_A(\xi,r):\xi\in\bR^d,r>0\} \end{equation}
denote the collection of all such balls. We notice that $\cB_A$ is invariant under affine transformations of the type $T_\eta:=\delta_\eta\cdot+c$ with $\eta>0$ and $c\in\bR^d$. In fact, 
one easily verifies that
\begin{equation}\label{eq:aff}
T_\eta {B}_A(\xi,r)={B}_A(\delta_\eta\xi+c,r\eta).    
\end{equation}
In particular, for $r>0$ and $\xi\in \bR^d$, we have
\[{B}_A(\xi,r)=T_r {B}_A(0,1)\]
with $T_r=\delta_r\cdot+\xi$.

\subsection{Muckenhoupt conditions adapted to anisotropic balls}\label{s:a}

The ``standard'' matrix Muckenhoupt condition on scalar and matrix weights, in a multivariate setting, is based on the geometry of Euclidean balls, see, e.g., \cite{Roud:2003a,FrazRoud:2004a}. A consequence of this is that the Muckenhoupt conditions on scalar and matrix weights are not automatically compatible with the anisotropic geometry defined using the quasi-norm from Definition \ref{def:A}, unless the quasi-norm happens to be equivalent to the Euclidean norm. Below we address this issue by considering modified Muckenhoupt conditions adapted to the quasi-norm given in Definition \ref{def:A}. Let us first consider scalar weights.

For a measurable function $f:\cD\rightarrow \bR$ and a subset $E\subset \bR^d$ of positive measure, 
we shall use the notation
$$\fint_E f(x)\,\d x:=\frac{1}{|E|}\int_E f(x)\,\d x\quad\text{and}\quad f(E):=\int_E f(x)\,\d x.$$
\begin{definition}
Suppose $w:\cD\rightarrow (0,\infty)$ is a measurable weight, and let $1\leq p<\infty$.
\begin{itemize}
\item In case $1<p<\infty$, we say that $w$ satisfies  the
 Muckenhoupt $A_p(\cB_A)$-condition, $1<p<\infty$, provided 
 \begin{equation}\label{eq:AP}
[w]_{A_p(\cB_A)}:=\sup_{E\in\cB_A}\fint_E w(x)\,\d x\cdot\bigg[\fint_E w^{-\frac{1}{p-1}}(x)\,\d x\bigg]^{p-1}<\infty,
\end{equation}
\item 
For $p=1$, we say that $w$ satisfies the
 Muckenhoupt $A_1(\cB_A)$-condition provided
\begin{equation}\label{eq:A1}
[w]_{A_1(\cB_A)}:=\sup_{E\in\cB_A}\fint_E w(x)\,\d x\cdot \esssup_{x\in E} w^{-1}(x)<\infty.
 \end{equation}
\end{itemize}
We denote the class of such scalar Muckenhoupt weights by $A_p(\cB_A)$.
\end{definition}
Next, we consider a matrix-valued function $W\colon \cD\rightarrow \bC^{N\times N}$, which is measurable
and strictly positive definite almost everywhere. We will refer to such a function as a \textit{matrix weight}. Inspired by the definition of Muckenhoupt $A_p$-conditions for matrix weights relative to standard Euclidean balls considered in \cite{Roud:2003a,FrazRoud:2004a}, which were based on the Muckenhoupt conditions first derived in \cite{TreiVolb:1997a,Volb:1997a,NazaTrei:1996a}, we follow \cite{NielSiki:2025a} and give the following definition of the Muckenhoupt $A_p$-condition for matrix weights.

\begin{definition}
Suppose $W\colon \cD\rightarrow \C^{N\times N}$ is a matrix weight and let $0<p<\infty$.
\begin{itemize}
    \item 
For $0<p\leq 1$,  $W$ is said to satisfy  the matrix Muckenhoupt  $A_p$-condition relative to $\cB_A$ provided 
\begin{equation}\label{eq:mA1}
	[W]_{{\mathbf{A}_p(\cB_A)}}:=\sup_{B\in\cB_A} \esssup_{t\in B} \fint_B \|W^{1/p}(x)W^{-1/p}(t)\|^p\dx<\infty.
\end{equation}
\item
In case $1<p<\infty$,  $W$ is said to satisfy  the matrix Muckenhoupt  $A_p$-condition relative to $\cB_A$ provided 
\begin{equation}\label{eq:Roudenko}
 [W]_{{\mathbf{A}_p(\cB_A)}}:=\sup_{B\in \cB_A} \bigg[\fint_B\left( \fint_B \big\|W^{1/p}(x)W^{-1/p}(t)\big\|^{p'} {\dt}\right)^{p/p'} {\dx}\bigg]^{1/p}<\infty,
\end{equation}
with $p'$ the dual exponent to $p$, i.e., $1/p+1/{p'}=1$. 
The norm $\|\cdot\|$ appearing in the integrals in Eqs.\ \eqref{eq:mA1} and \eqref{eq:Roudenko} is any matrix norm on the $N\times N$ matrices. We write $W\in \mathbf{A}_p(\cB_A)$ whenever $[W]_{{\mathbf{A}_p(\cB_A)}}<\infty$, $0<p<\infty$.
\end{itemize}
\end{definition}

\begin{remark}\leavevmode\label{re:dual}
 It is true that $W\in\mathbf{A}_p(\cB)$, $1<p<\infty$, implies $W^{-p'/p}\in \mathbf{A}_{p'}(\cB)$ with $p'$ the dual exponent to $p$, see, e.g., \cite{Roud:2003a} for the proof in the Euclidean case. The reader can verify that the proof in \cite{Roud:2003a} can easily be adapted to the case considered here. 
\end{remark}

For $0< p<\infty$, it is possible to create scalar Muckenhoupt weights from matrix weights satisfying a Muckenhoupt condition as stated in the following result, which we will need in the sequel. 

\begin{lemma}\label{le:sc}
Let $0<p<\infty$ and suppose $W\colon \cD\rightarrow \bC^{N\times N}$ is in $\mathbf{A}_p(\cB_A)$. Then there is a constant $C:=C(p)$ such that
\begin{itemize}
    \item[(i)] for any $\vx\in\bC^N\backslash \{\mathbf{0}\}$, the scalar weight $w_\vx(t):=|W^{1/p}(t)\vx|^p$ is in ${A}_p(\cB_A)$ and $[w_\vx]_{  {A}_p(\cB_A)}\leq C [W]_{  \mathbf{A}_p(\cB_A)}$ provided $1<p<\infty$,
    \item[(ii)] for any $\vx\in\bC^N\backslash \{\mathbf{0}\}$, the scalar weight $w_\vx(t):=|W^{1/p}(t)\vx|^p$ is in ${A}_1(\cB_A)$ and $[w_\vx]_{  {A}_1(\cB_A)}\leq  [W]_{  \mathbf{A}_p(\cB_A)}$ provided $0<p\leq 1$.
\end{itemize}
The estimates in (i) and (ii) also hold true for the scalar weight $v(t):=\|W^{1/p}(t)\|^p$ in place of $w_\vx$.
\end{lemma}

We refer to \cite{mn_geo} for a proof. A consequence of Lemma \ref{le:sc} is that the weight $w_\vx$ satisfies the following doubling condition.
 
\begin{lemma}\label{le:doub}
 Let $0<p<\infty$ and suppose $W\in\mathbf{A}_p(\cB_A)$.  For $\vx\in\bC^N\backslash \{\mathbf{0}\}$, the weight $w_\vx(\cdot):=|W^{1/p}(\cdot)\vx|^p$ satisfies the doubling condition
   \begin{equation}\label{eq:Wdoub}
   \int_{B_{A}(\vy,2r)} w_\vx(t)\,\dt\leq C\int_{B_{A}(\vy,r)} w_\vx(t)\,\dt,
    \end{equation}
         with a constant $C:=C([W]_{\mathbf{A}_p(\cB_A)})$ independent of $\vx\in\bC^N\backslash \{\mathbf{0}\}$, $\vy\in\bR^d$, and $r>0$.
\end{lemma}
Lemma \ref{le:sc} was first proved by Calder\'on \cite{Cald:1976a} in a setting that covers the geometry of anisotropic balls $\cB_A$ considered in the present paper. We refer to \cite{mn_geo} for further details.


\section{Matrix Weighted Smoothness Spaces of Decomposition type}\label{sec:3}
In this section, we define the main class of matrix-weighted smoothness spaces compatible with the anisotropy from Definition \ref{def:A}. The spaces are constructed using the so-called decomposition method on the Fourier side, where the spaces are defined by imposing suitable weighted $L^p$-restrictions on local components of a vector-valued function. The local components will be defined using general flexible decompositions of the frequency space.

Roudenko was the first to apply such an approach in the matrix-weighted setting, see \cite{Roud:2003a}, where she introduced a very natural notion of matrix-weighted Besov spaces $B^s_{p,q}(W)$ based on dyadic decompositions of vector-valued functions. This work was later extended by Frazier and Roudenko  \cite{FrazRoud:2004a, FrazRoud:2021a} to matrix-weighted Triebel-Lizorkin spaces and Besov spaces for the full range $0<p<\infty$.

In order to decompose the frequency space in a flexible manner, while retaining a suitable structure, we first introduce so-called admissible coverings and show how to generate them. These coverings are then used to construct a suitable resolution of unity and then define Triebel-Lizorkin-type smoothness spaces and associated modulation spaces. Finally, we construct a frame which will be used in the following sections to generate compactly supported frame expansions.
\begin{definition}\label{def:adm}
A family $\mathcal{C}:=\{C_k\}_{k\in \Z^d}$ of bounded measurable subsets
$C_k\subset \bR^d$ is called an admissible covering if
$\bR^d=\cup_{k\in \Z^d}C_k$ and there exists $n_0<\infty$ such that
$\#\{j\in \Z^d:C_k\cap C_j\not= \emptyset\}\le n_0$ for all $k\in
\Z^d$.
\end{definition}

\begin{remark}
    Clearly, one may use any countable infinite set as an index set in place of $\bZ^d$ in Definition \ref{def:adm}, but we prefer to think of the index $k$ as being related to a suitably sampled frequency from $\bR^d$. 
\end{remark}

We will mainly be interested in the case where the covering is formed by a suitable collection
of $|\cdot|_A$-balls, where the radius of a given ball is a so-called moderate
function of its center.
\begin{definition}\label{entertainus}%
A function $h:\bR^d\to (0,\infty)$  is
called ($|\cdot|_A$-)moderate if there exist constants $\rho_0, R_0>0$ such that
$|\xi-\zeta|_A\le \rho_0 h(\xi)$ implies $R^{-1}_0\le
h(\zeta)/h(\xi)\le R_0$, and $h$ satisfies the growth condition
\begin{equation}\label{eq:growthh}
    C_1\brac{\xi}_A^\tau\leq h(\xi)\leq C_2\brac{\xi}_A,\qquad \xi\in\bR^d,
\end{equation}
for some $\tau, C_1, C_2>0$. 
\end{definition}

\begin{remark}
Since we consider the quasi-distance $|\cdot|_A$ fixed, we will slightly abuse notation and refer to a 
$|\cdot|_A$--moderate function simply as a moderate function.
\end{remark}

\begin{example}\label{cantexplain}\noindent
Let $0< \eta \le 1$. Then
\begin{equation}
h_\eta(\xi):=(1+|\xi|_A)^\eta
\end{equation}
is moderate.
\end{example}
With a moderate function $h$, it is then possible to construct an
admissible covering by using balls (see \cite[Lemma 4.7]{Feic:1987a} and \cite[Lemma 5]{BoruNiel:2008a}):
\begin{lemma}\label{braunsugar}\label{le:Q}\noindent
Given a moderate function $h$ with constants $\rho_0,R_0>0$, there
exists a countable admissible covering
$\mathcal{C}:=\{B_A(\xi_k,\rho h(\xi_k))\}_{k\in\Z^d}$ for
$\rho < \rho_0$, such that
\begin{itemize}
   \item $\{B_A(\xi_k,\frac{\rho}{2} h(\xi_k))\}_{k\in\Z^d}$ is also an admissible covering, and
    \item there exists a constant
   $0<\rho'<\rho/2$ such that the sets
   $\{B_A(\xi_k,\rho' h(\xi_k))\}_{k\in\Z^d}$
     are pairwise
    disjoint.
\end{itemize}
\end{lemma}
\begin{remark}
 Using the estimate \eqref{natasha}, we may choose  $\rho$ small enough such that $0<\rho\leq 1$ and  $B_A(0,\rho)\subset [-\pi,\pi]^d$. This scaling will be convenient for the definition of a tight frame in Eq.\ \eqref{eq:ltf}.
\end{remark}
 Notice that the covering $\mathcal{C}$ from Lemma
\ref{braunsugar} is generated by a family of invertible affine
transformations applied to $B_A(0,\rho)$ in the sense
that
\begin{equation*}
B_A(\xi_k,\rho h(\xi_k))=T_kB_A(0,\rho), \,\,
T_k:=\delta_{r_k}\cdot+\xi_k,
\end{equation*}
where we put 
\begin{equation}\label{eq:rk}
    r_k:= h(\xi_k).
\end{equation}

An
important property of the covering we will need is that whenever
$B_A \big(\xi_j,\rho h(\xi_j)\big)\cap B_A
\big(\xi_k,\rho h(\xi_k)\big)\not=\emptyset$ then $h(\xi_j)\asymp
h(\xi_k)$ uniformly in $j$ and $k$, which follows from the fact that
$h$ is moderate and $\rho<\rho_0$. We deduce that there
exists a uniform constant $K$ such that
\begin{equation}
  \label{eq:esc}
  \|\delta_{h(\xi_k)}^{-1}\delta_{ h(\xi_j)}\|_{\ell_2(\bR^{d\times d})}\leq K\qquad 
\text{whenever}\quad B_A \big(\xi_j,\rho h(\xi_j)\big)\cap B_A
\big(\xi_k,\rho h(\xi_k)\big)\not=\emptyset.
\end{equation}

We are now in a position to generate a suitable resolution of unity that, for technical reasons, must satisfy the following conditions.
\begin{definition}\label{def:bapu}
Let $\mathcal{C}:= \{T_kB_A(0,\rho)\}_{k\in\Z^d}$  be an
admissible covering of $\bR^d$ from Lemma \ref{braunsugar}. A
corresponding bounded admissible partition of unity (BAPU) is a
family of functions $\Phi=\{\phi_k\}_{k\in\Z^d}\subset
\mathcal{S}$ satisfying:
\begin{itemize}
\item[(i)] $ \textrm{supp}(\phi_k)\subseteq T_kB_A(0,\rho),\,
k\in\Z^d$.
\item[(ii)] $\sum_{k\in\Z^d} \phi_k(\xi)=1,\, \xi \in \bR^d$.
\item[(iii)] For $M>0$ there is a constant $K:=K(M)$ such that, for $k\in\bZ^d$,
$$|\mathcal{F}^{-1}(\phi_k)(x)|\leq K r_k^{\nu} (1+r_k|x|_A)^{-M},\qquad x\in\bR^d.$$
\end{itemize}

\end{definition}
A standard procedure for generating a BAPU for $\mathcal{C}$ is to pick
$g \in C^\infty(\bR^d)$ non-negative with
$\textrm{supp}(g)\subseteq B_A(0,\rho)$ such that there is $\delta>0$ satisfying
$g(\xi)\geq \delta$ for $\xi \in B_A(0,\rho/2)$. 
We define
\begin{equation}
g_k(\xi):=g(T_k^{-1}).
\end{equation}
One can show
that
\begin{equation*}
\phi_k(\xi):=\frac{g_k(\xi)}{\sum_{j\in\Z^d}g_j(\xi)}
\end{equation*}
defines a BAPU for $\mathcal{C}$; see \cite[Section 4]{BoruNiel:2008a}. For later use, we also consider
\begin{equation}\label{yesterday}
\psi_k(\xi):=\frac{g_k(\xi)}{G(\xi)},
\end{equation}
with 
\begin{equation}\label{eq:G}
    G(\xi):=\sqrt{\sum_{j\in\Z^d}g_j^2(\xi)}.
    \end{equation}
 Notice that  $\Psi:=\{\psi_k\}_{k\in\bZ^d}$ defines a ``square-root'' of a BAPU in the sense that
 $$\sum_{k\in\Z^d}\psi_k^2(\xi)=1,\qquad \xi\in\bR^d.$$
 

\begin{remark}\label{rem:b}
For a BAPU $\{\phi_k\}_{k\in \bN}$ associated 
with the admissible covering $\{T_kB_A(0,\rho)\}_{k\in
  \bN}$ we define 
  \begin{equation}\label{eq:kk2}
{\phi}_k^{*}
:=\sum_{j\in \tilde{k}} \phi_j,
\end{equation}
where $\tilde{k}:=\{j\in\bZ^d: \text{supp}(\phi_j)\cap T_kB_A(0,\rho)\not=\emptyset\}$.
We will use extensively that $\phi_k\phi_k^*=\phi_k$ and that there exists $M$ independent of $k$ such that  $\#\tilde{k}\leq M$, see \cite{FeicGrob:1985a,Feic:1987a}. 
\end{remark}

With a BAPU in hand, we can now define the associated matrix-weighted modulation spaces. 
We denote by $\cS=\cS(\bR^d)$ the Schwartz space of rapidly decreasing, infinitely differentiable functions on $\bR^d$. A function $\varphi\in\cC^{\infty}$ belongs to $\cS(\bR^d)$ when, for every $k\in\bN_0$ with $\bN_0:=\bN\cup\{0\}$, the seminorms
\begin{equation}\label{Snorm1}
p_{k}(\varphi):=\max_{\alpha\in \bN_0^d: |\alpha|\leq k}\sup_{x\in\bR^d}(1+|x|)^{k}|\partial ^\alpha \varphi(x)|
\end{equation}
are all finite, where we put $|\alpha|:=\sum_{j=1}^d \alpha_j$ for $\alpha\in \bN_0^d$. As is well known, the seminorms $\{p_k\}$ turn $\cS$ into a Fr\'{e}chet space.
The dual space $\cS'=\cS'(\bR^d)$ of $\cS$ is the space of tempered distributions. It will also be useful to consider the corresponding 
concepts in a vector-valued setting, where we consider the direct-sum Fr\'{e}chet space $\bigoplus_{j=1}^N\dS$, with 
dual space $\bigoplus_{j=1}^N\mathcal{S}'(\bR^d)$ consisting of $N$-tuples of tempered distributions.

 Let $m:\bR^d\rightarrow\bC$ be a bounded measurable function (a multiplier). We denote by $m(D)f:=\mathcal{F}^{-1}(m\hat{f})$ the corresponding Fourier multiplier operator, i.e., the convolution of $\mathcal{F}^{-1}(m)$ with $f$.

We are now ready to give the definition of the vector-valued weighted decomposition spaces.
\begin{definition}\label{def:mod}
Let $W:\Rd\rightarrow \bC^{N\times N}$ be a matrix weight, and let $\cC=\{C_k\}_{k\in\bZ^d}$ be an admissible covering of $\bR^d$ with associated BAPU 
   $\Phi:=\{\phi_k\}_{k\in\bZ^d}$ of the type considered in Definition \ref{def:bapu}. 
For  $s\in\bR$, $0<p<\infty$, and $0<q\leq \infty$, we let $M^{s}_{p,q}(W)$ denote the collection of all vector-valued distributions 
$\vf=(f_1,\ldots,f_N)^T\in\bigoplus_{j=1}^N\mathcal{S}'(\bR^d) $, such that
$$\|\vf\|_{M^{s}_{p,q}(W)}:=\bigg\|\bigg\{|C_k|^{s/\nu}\|\phi_k(D)\vf\|_{L^p(W)}\bigg\}_{k\in\bZ^d}\bigg\|_{\ell_q}<\infty,$$
with $\phi_k(D)\vf:=(\phi_k(D)f_1,\ldots,\phi_k(D)f_N)^T$ acting coordinate-wise. For $q=\infty$, the $\ell^q$-norm is replaced by the supremum over $C$.
\end{definition}
\begin{remark} \,\vskip 3pt
    \begin{itemize}
        
        \item[(a)] Notice that $M^{s}_{p,q}(W)$ a priori depends on both the choice of covering $\cC$ and BAPU $\Phi$. Proposition \ref{prop:complete} below will show that up to equivalence of (quasi-)norms, $M^{s}_{p,q}(W)$ depends only on the moderate function $h$.
    \item[(b)] We recall that by construction $C_k=\delta_{r_k}B_A(0,\rho)+\xi_k$, so consequently $|C_k|\asymp r_k^{\nu}$ and we may replace the weight $|C_k|^{s/\nu}$ in Definition \ref{def:mod} by $r_k^s$ whenever convenient. 
    \end{itemize}
\end{remark}
It turns out that $M^{s}_{p,q}(W)$ is in fact a (quasi-)Banach space, at least for matrix weights $W\in \mathbf{A}_p(\cB_A)$. We have the following result, where it is also shown that up to equivalence of norms, $M^{s}_{p,q}(W)$ is independent of the choice of BAPU whenever $W\in\mathbf{A}_p(\cB_A)$.
\begin{proposition}\label{prop:complete}
Let $0< p<\infty$ and suppose $W\in \mathbf{A}_p(\cB_A)$. For $0<q\leq \infty$ and $s\in\bR$, the following statements hold:
\begin{itemize}
\item[(a)] We have continuous embeddings
$$\bigoplus_{j=1}^N\dS\hookrightarrow M^{s}_{p,q}(W)\hookrightarrow \bigoplus_{j=1}^N\mathcal{S}'(\bR^d).$$
    \item[(b)] The space $M^{s}_{p,q}(W)$ is complete, i.e., $M^{s}_{p,q}(W)$ is a (quasi-)Banach space.
    \item[(c)] The space $M^{s}_{p,q}(W)$ is independent of the choice of BAPU up to equivalence of (quasi-)norms.
\end{itemize}

\end{proposition}

We will postpone the proof of (a) and (b) until Appendix \ref{sec:complete} as we first need to develop a number of technical tools providing estimates to handle certain band-limited vector-valued functions.
To prove (c), we will need the following convolution lemma proved 
by the author in \cite{mn_geo}, extending an earlier
result by Frazier and Roudenko \cite[Lemma 4.4]{FrazRoud:2021a} to the present anisotropic setting.

For $B=B_A(c,R)\in \cB_A$, we define the following class of vector-valued functions with band-limited coordinate functions:
\begin{equation}
E_B:=\{\vf:\bR^d\rightarrow \bC^N: f_i\in \ddS \text{ and } \text{supp}(\hat{f}_i)\subseteq B, i=1,\ldots,N\}.
\end{equation}

\begin{proposition}\label{prop:main}
Let $0<p<\infty$ and suppose $W\in\mathbf{A}_p(\cB_A)$. Let $B=B_A(c,R)\in\cB_A$ and put $\beta:=\max\{\nu,\nu p\}$.  Suppose there is a constant $K$ such that the compactly supported function $\phi:B_A(c,R)\rightarrow\bC$ satisfies 
\begin{equation}\label{eq:e1}
|\mathcal{F}^{-1}(\phi)(x)|\leq K R^{\nu} (1+R|x|_A)^{-M},\qquad x\in\bR^d, 
\end{equation}
for some $$M>\max\big\{\nu+\max\{0,p-1\}\beta,(\nu+\beta)/\min\{1,p\}\big\}.$$  Then there exists a finite constant $C:=C([W]_{\mathbf{A}_p(\cB_A)},K,p)$ such that the Fourier multiplier
$$\phi(D)f:=\mathcal{F}^{-1}[\phi\cdot \mathcal{F}(f)],$$ defined for $f\in \ddS $ with $\text{supp}(\hat{f})\subseteq B_A(c,R)$, satisfies
$$\|\phi(D)\vf\|_{L^p(W)}\leq C \|\vf\|_{L^p(W)}$$
for all $\vf\in E_B\cap L^p(W)$.
\end{proposition}

We now give a proof of Proposition \ref{prop:complete}.(c).
\begin{proof}[{Proof of Proposition \ref{prop:complete}.(c)}]
Let $\cC=\{C_k\}_{k\in\bZ^d}$ and $\cP=\{P_k\}_{k\in\bZ^d}$ be any two  structured coverings with associated BAPUs
$\Phi=\{\phi_k\}_{k\in\bZ^d}$ and $\Gamma=\{\gamma_k\}_{k\in\bZ^d}$, respectively, of the type considered in Definition \ref{def:bapu}. First, suppose $\Gamma$ is used for Definition \ref{def:mod} and let $\vf\in M^s_{p,q}(W)$.
We first notice, using the uniformly bounded height of any structured covering, that for $R\in \cR$,
\begin{equation}\label{eq:pdo}
\phi_{k}(D)\vf=\phi_{k}(D)\sum_{j\in F_{k}}\gamma_{j}(D)\vf,
\end{equation}
with $F_{k}=\{j\in \bZ^d:P_j\cap C_k\not=\emptyset\}$, where we recall that $\#F_{k}$ is bounded by a constant $n_0$ independent of $k$, see Definition \ref{def:adm}. Hence, by Proposition \ref{prop:main} and the decay condition from (iii) in Definition \ref{def:bapu},
$$\|\phi_{k}(D)\vf\|_{L^p(W)}\leq c\sum_{j\in F_{k}}\|\gamma_{j}(D)\vf\|_{L^p(W)}.$$
By the moderateness of $h$ and the construction of the structured coverings, whenever $C_k\cap P_j\not=\emptyset$, the corresponding scales are comparable; hence $|C_k|\asymp |P_j|$. It follows from this observation that
$$|C_k|^{s/\nu}\|\phi_{k}(D)\vf\|_{L^p(W)}\leq c'\sum_{j\in F_{k}}|P_j|^{s/\nu}\|\gamma_{j}(D)\vf\|_{L^p(W)}.$$

Using the uniform bounds on the cardinality of the set $F_{k}$, it is then straightforward to verify that
\begin{align*}
\bigg\|\bigg\{|C_k|^{s/\nu}\|\phi_{k}(D)\vf\|_{L^p(W)}\bigg\}_{k\in\bZ^d}\bigg\|_{\ell_q}&\lesssim
\bigg\|\bigg\{|P_j|^{s/\nu}\|\gamma_{j}(D)\vf\|_{L^p(W)}\bigg\}_{j\in\bZ^d}\bigg\|_{\ell_q}\asymp \|\vf\|_{M^s_{p,q}(W)}.
\end{align*}
By interchanging the roles of $\cC$ and $\cP$, we derive the reverse estimate and the claimed equivalence follows.
\end{proof}

\begin{remark}\label{re:sqrt}
    Using the same argument as in the proof of Proposition \ref{prop:complete}.(c), one can easily verify that
   $$ \bigg\|\bigg\{|C_j|^{s/\nu}\|\psi_j(D)\vf\|_{L^p(W)}\bigg\}_{j\in\bZ^d}\bigg\|_{\ell_q}$$
    yields an equivalent norm on $M^s_{p,q}(W)$ for $0<p<\infty$, $0<q\leq \infty$, where $\Psi=\{\psi_k\}_k$ is the modified ``square-root'' BAPU defined in Eq.\ \eqref{yesterday}.
    \end{remark}

\section{Discrete vector valued modulation spaces and norm characterizations}\label{sec:ms}
Often the notion of smoothness can be linked to a sparsity condition on a suitable expansion coefficient sequence for the function. In this section, we define a discrete vector-valued weighted modulation-type space together with a simple construction of adapted tight frames that will support a $\phi$-transform in the spirit of the classical construction by Frazier and Jawerth \cite{FrazJawe:1985a,FrazJawe:1990a}.

Let $\cC=\{C_k\}_{k\in\bZ^d}$ be an admissible covering of $\bR^d$ with associated BAPU 
   $\{\psi_k\}_{k\in\bZ^d}$ of the type considered in Definition \ref{def:bapu}. We have $C_k=\delta_{r_k}B_A(0,\rho)+\xi_k$ for suitable $r_k>0$ and $\xi_k\in\bR^d$. We now define a family of ``cubes'' $\cQ:=\cQ(\cC)$ associated with $\cC$,
   that will play the role of a substitute for the standard dyadic cubes.
   
 Using the positive parameters $\{r_k\}_{k\in\bZ^d}$ obtained from $\cC$, we define the sets
\begin{equation}\label{eq:qk}
Q(k,\ell):=\bigg\{y\in\bR^d:\delta_{r_k}y-\ell\in B_A(0,1)\bigg\}.
\end{equation}
For fixed $k$, we put $\cQ_k:=\cup_\ell Q(k,\ell)$, and uniformly in $k$ and $\ell$, we have $|Q(k,\ell)|\asymp r_k^{-\nu}$.
It is easy to verify that there exists $L<\infty$ so that uniformly in $x$ and $k$,  
\begin{equation}\label{eq:L}
\sum_{\ell\in\bZ^d}\mathbf{1}_{Q(k,\ell)}(x)\leq L,   
\end{equation}
where $\mathbf{1}_A$ denotes the characteristic function of a measurable set $A$.
For notational convenience, let
\begin{equation}\label{eq:xk}
    x_{k,\ell}:=\delta_{r_k^{-1}}\ell,
\end{equation}
and notice that 
\begin{equation}\label{eq:Qkk}
Q(k,\ell)=B_A\big(x_{k,\ell},r_k^{-1}\big).
\end{equation}
 Finally, denote the full family by  
\begin{equation}\label{eq:Q}
\cQ=\bigcup_{k\in\bZ^d} \cQ_k.    
\end{equation}

We now consider the following definition.

\begin{definition}\label{def:eqL}
Let $W:\Rd\rightarrow \bC^{N\times N}$ be a matrix weight, and let $\cC=\{C_k\}_{k\in\bZ^d}$  be an admissible covering of $\bR^d$ with associated BAPU 
   $\{\psi_k
   \}_{k\in\bZ^d}$ of the type considered in Definition \ref{def:bapu}. We let 
 $\cQ=\{Q(k,\ell)\}_{k,\ell}$ be the collection of sets defined in \eqref{eq:qk}.
We let $m^{s}_{p,q}(W)$ denote the collection of all vector-valued sequences 
$\vs=\{\vs_{k,\ell}\}_{{k,\ell}\in\bZ^d}$ such that
\begin{align*}
\|\{\vs_{k,\ell}\}_{{k,\ell}}\|_{m^{s}_{p,q}(W)}&:=\bigg\|\bigg\{r_k^s\bigg(\sum_{\ell\in\bZ^d}\big\| |Q(k,\ell)|^{-\frac{1}{2}}\vs_{k,\ell}\mathbf{1}_{Q(k,\ell)}\big\|_{L^p(W)}^p\bigg)^{1/p}\bigg\}_{k\in\bZ^d}\bigg\|_{\ell_q}.
\end{align*}
 For $q=\infty$, the $\ell^q$-norm is replaced by the supremum over $k$.
\end{definition}

Let us now consider a simple construction of a tight frame compatible with the decomposition provided by $\cC$. Consider the modified BAPU
$\{\psi_k\}_{k\in\bZ^d}$ given by \eqref{yesterday} associated with the
admissible covering $\mathcal{C}=\{T_k\mathcal{B}_A(0,\rho)\}_{k\in \bN}$
generated by $\{T_k=\delta_{r_k}\cdot+\xi_k\}_{k\in\bZ^d}$.  Recall that we have the scaling ${B}_A(0,2\rho)\subseteq K_\pi$
with $K_\pi:=[-\pi,\pi]^d$. Then we
consider the orthonormal system
$$e_{k,n}(\xi) :=
(2\pi)^{-\frac{d}2}r_k^{-\nu/2}\mathbf{1}_{K_\pi}(T_k^{-1}\xi)e^{in\cdot
  T_k^{-1}\xi},\qquad n,k\in \bZ^d,$$
and define
\begin{equation}
  \label{eq:ltf}
  \hat\psi_{k,n} :=\psi_ke_{k,n}= (2\pi)^{-\frac{d}2}r_k^{-\nu/2}\psi_n(\cdot)e^{in\cdot
  T_k^{-1}\cdot},\qquad n,k\in \bZ^d,
\end{equation}
where we used that $\text{supp}(\psi_k)\subseteq T_kB_A(0,\rho)$.
It is straightforward to verify that
$\{\psi_{k,n}\}_{k,n}$ is a tight frame for $L_2(\bR^d)$; for further details, see, e.g., \cite[Section 6]{BoruNiel:2008a}.

We would like to estimate the frame coefficients for any $\vf\in M_{p,q}^{s}(W)$. Considering the band-limited nature of the system from \eqref{eq:ltf}, it is not surprising that one way to obtain such an estimate is to invoke a suitable sampling result adapted to the covering $\cC$. The following proposition was proved by the author in \cite{mn_geo}.

\begin{proposition}\label{prop:samp}
    Let $0<p<\infty$ and suppose $W\in\mathbf{A}_p(\cB_A)$, and  put $B=B_A(c,r_k)\in\cB_A$. 
Then there exists a constant $c_{p,d}$ such that for $\vg\in E_B\cap L^p(W)$,
$$\sum_{\ell\in \bZ^d}\int_{Q_{(k,\ell)}}\big|W^{1/p}(x)\vg\big(\delta_{r_k}^{-1}\ell\big)\big|^p\,dx\leq c_{p,n}\|\vg\|_{L^p(W)}^p,$$
where $Q(k,\ell)$ is defined in Eq.\ \eqref{eq:qk}.
\end{proposition}

By adapting Proposition \ref{prop:samp}, we can derive the following result.

\begin{proposition}\label{prop:coef}
 Let $W:\Rd\rightarrow \bC^{N\times N}$ be a matrix weight, and let $\cC=\{C_k\}_{k\in\bZ^d}$  be an admissible covering of $\bR^d$ with associated BAPU 
   $\{\psi_k
   \}_{k\in\bZ^d}$ of the type considered in Definition \ref{def:bapu}. Let $0< p<\infty$, $0<q\leq\infty$, and suppose $W\in \mathbf{A}_p(\cB_A)$. Then there exists a constant $C:=C(p,q,W,\mathcal C)$ such that for $\vf\in M_{p,q}^{s}(W)$,
  \begin{equation}
      \|\{\vc_{k,\ell}\}_{k,\ell}\|_{m_{p,q}^{s}(W)}\leq C \|\vf\|_{M_{p,q}^{s}(W)},
  \end{equation}
  with $\vc_{k,\ell}:=\langle \vf,\psi_{k,\ell}\rangle$.
\end{proposition}
\begin{proof}
Recall that $\psi_{k,\ell}\in E_{C_k}$ with
$$\hat{\psi}_{k,\ell}(\xi)=
(2\pi)^{-{d}/2}r_k^{-\nu/2}\psi_k(\xi)e^{in\cdot
  T_k^{-1}\xi}
,
$$
so, by a direct calculation, we have
$$|W^{1/p}(x)\langle \vf,\psi_{k,\ell}\rangle|=(2\pi)^{-{d}/2}|Q(k,\ell)|^{1/2}\bigg|W^{1/p}(x)[\psi_k(D)\vf]\bigg(\delta_{r_k}^{-1}\ell\bigg)\bigg|.$$
Notice that $\text{supp}(\psi_k)\subseteq T_kB_A(0,\rho)\subseteq T_kB_A(0,1)$, so we may call on Proposition \ref{prop:samp} to obtain the following estimate,
\begin{align*}
  \|  \{\langle \vf,\psi_{k,\ell}\rangle\}_{k,\ell} \|_{m_{p,q}^{s}(W)}&=
  \big\|\big\{r_k^s\sum_{\ell\in\bZ^d}\big\| |Q(k,\ell)|^{-\frac{1}{2}}\langle \vf,\psi_{k,\ell}\rangle\mathbf{1}_{Q(k,\ell)}\big\|_{L^p(W)}\big\}_k\big\|_{\ell_q}\\
  &\asymp\bigg\|\bigg\{r_k^s\bigg[\sum_{\ell\in \bZ^d}\int_{Q(k,\ell)}\big|W^{1/p}(x)[\psi_k(D)\vf] \big(r_k^{-1}\ell\big)\big|^p dx\bigg]^{1/p}\bigg\}_k\bigg\|_{\ell_q}\\
  &\leq c_{p,q}\big\|\big\{r_k^s \|\psi_k(D)\vf\|_{L^p(W)}\big\}_k\big\|_{\ell_q}\\
  &\asymp \|\vf\|_{M_{p,q}^{s}(W)},
\end{align*}
where we have used Definition \eqref{def:eqL} and the observation from Remark \ref{re:sqrt} for the final estimate.
\end{proof}

  We turn our attention to the canonical reconstruction/synthesis operator $T$ for the tight frame $\{\psi_{j,\ell}\}$.  The reconstruction/synthesis operator $T$ is defined for (finite) sequences $\vc=\{\vc_{j,\ell}\}$ by $$T\vc:=\sum_{{j,\ell}}\vc_{j,\ell} \psi_{j,\ell}.$$ 
  We will prove that $T$ is bounded from $m_{p,q}^{s}(W)$ to $M_{p,q}^{s}(W)$ for suitable weights $W$.

\begin{remark}\label{rem:T}
The fact that $\{\psi_{j,\ell}\}$ is a redundant system in $L^2(\mathbb{R}^d)$ makes it a nontrivial issue to verify that $T$ is actually \textit{well-defined} with an extension to all of $m_{p,q}^{s}(W)$ based on the result for finite sequences presented in Proposition \ref{prop:recon} below. We shall not pursue the details here, but rather refer the reader to the general approach to the issue presented in \cite[Section 3]{GeorJohnNiel:2017a}, which is applicable to the present setup.
\end{remark}

The doubling condition for a matrix weight will play a central role. Using the notation introduced in Section \ref{sec:HS}, we have the following definition.

\begin{definition}\label{def:doub}
    We say that the matrix weight $W:\cD\rightarrow \bC^{N\times N}$ satisfies the doubling condition of order $0<p<\infty$ if there is a constant $c$ such that
    for all $x\in\cD$, $\vy\in\bC^N$, and $r>0$,
    \begin{equation}\label{eq:doub}
        \int_{B_A(x,2r)} |W^{1/p}(t)\vy|^p\,dt\leq c\int_{B_A(x,r)} |W^{1/p}(t)\vy|^p\,dt.
    \end{equation}
    Let $c =2^\beta$, with $\beta\geq \nu$, be the smallest constant for which \eqref{eq:doub} holds, then $\beta$ is called the doubling exponent of $W$.
\end{definition}

\begin{remark}\label{rem:doub}
 It follows directly from Lemma \ref{le:doub}
that Eq.\ \eqref{eq:doub} is always satisfied whenever $W \in \mathbf{A}_p(\cB_A)$, $0<p<\infty$.
 \end{remark} 
 
 We have the following result that, in particular, applies to any matrix weight in $\mathbf{A}_p(\cB_A)$, cf.\ Remark \ref{rem:doub}. 
\begin{proposition}\label{prop:recon}
Let $W:\Rd\rightarrow \bC^{N\times N}$ be a matrix weight, and let $\cC=\{C_k\}_{k\in\bZ^d}$  be an admissible covering of $\bR^d$ with associated BAPU 
   $\{\phi_k
   \}_{k\in\bZ^d}$ of the type considered in Definition \ref{def:bapu}. Let $0< p<\infty$, $0<q\leq \infty$, and suppose $W$ satisfies the doubling condition in Eq.\ \eqref{eq:doub}. Then there exists a constant $C:=C(p,q,W)$ such that for any finite vector-valued coefficient sequence  $\vs:=\{\vc_{j,\ell}\}_{(j,\ell)\in F}$, $F\subset \bZ^d\times\bZ^d$,
  \begin{equation}
      \bigg\|\sum_{(j,\ell)\in F}\vc_{j,\ell}\psi_{j,\ell}\bigg\|_{M_{p,q}^{s}(W)}\leq C \|\{\vc_{j,\ell}\}\|_{m_{p,q}^{s}(W)}.
  \end{equation}
  
\end{proposition}
\begin{proof}
We have, using the fact that $\text{supp}(\phi_k),\text{supp}(\hat{\psi}_{k,\ell})\subseteq C_k$,

\begin{align*}
      \bigg\|\sum_{(j,\ell)\in F}\vc_{j,\ell}\psi_{j,\ell}\bigg\|_{M_{p,q}^{s}(W)}&=
      \bigg\|\bigg\{r_k^s\bigg\|\phi_k(D)\sum_{(j,\ell)\in F}\vc_{j,\ell}\psi_{j,\ell}\bigg\|_{L^p(W)}\bigg\}_k\bigg\|_{\ell_q}\\
      &= \bigg\|\bigg\{r_k^s\bigg\|\phi_k(D)\sum_{j\in N_k}\sum_\ell\vc_{j,\ell}\psi_{j,\ell}\bigg\|_{L^p(W)}\bigg\}_k\bigg\|_{\ell_q},
\end{align*}
where $N_k=\{j\in\bZ^d:C_j\cap C_k\not=\emptyset\}$. Now, $r_j\asymp r_k$ (uniformly) for $j\in N_k$ by the moderation property of $h$, see Definition \ref{entertainus}. Hence, by Proposition \ref{prop:main},

\begin{align*}
r_k^s\bigg\|\phi_k(D)\sum_{j\in N_k}\sum_\ell\vc_{j,\ell}\psi_{j,\ell}\bigg\|_{L^p(W)}
     & \leq C r_k^s\bigg\|\sum_{j\in N_k}\sum_\ell\vc_{j,\ell}\psi_{j,\ell}\bigg\|_{L^p(W)}\\
     &\leq C' \sum_{j\in N_k}\bigg\|r_j^s \sum_\ell\vc_{j,\ell}\psi_{j,\ell}\bigg\|_{L^p(W)},
\end{align*}
where we used that $\# N_k$ is uniformly bounded in $k$.
First, suppose $1<p<\infty$. We now use that $\psi_{j,\ell}$ satisfies the decay property \eqref{eq:e1} for any $N>0$ to obtain the estimate

\begin{align*}
\bigg\| \sum_\ell\vc_{j,\ell}\psi_{j,\ell}\bigg\|_{L^p(W)}^p&\leq 
\int_{\bR^d} \bigg(\sum_\ell |W^{1/p}(x)\vc_{j,\ell}|\, |\psi_{j,\ell}(x)|\bigg)^p\,dx\\
&\leq C_N\int_{\bR^d} \bigg(r_j^{\nu/2}\sum_\ell |W^{1/p}(x)\vc_{j,\ell}|\, \big(1+r_j|x-x_{j,\ell}|_A\big)^{-N}\bigg)^p\,dx\\
&\leq C_N'r_j^{\nu p/2}\int_{\bR^d} \sum_\ell |W^{1/p}(x)\vc_{j,\ell}|^p\, \big(1+r_j\big|x-x_{j,\ell}\big|_A\big)^{-\frac{Np}{2}}\,dx,
    \end{align*}
where we used the discrete H\"older inequality for the last step with $N$ chosen large enough that, for the dual H\"older exponent $p'$ to $p$,
$$
  \sup_{u\in\bR^d} \sum_{\ell\in\bZ^d} \big(1+\big|u-\ell\big|_A\big)^{-\frac{Np'}{2}} <\infty,$$
where the existence of such $N$ can be deduced from the estimate in Eq.~\eqref{natasha}.
In the remaining case, $0<p\leq 1$, we obtain a similar estimate more directly by using sublinearity.

  The function $w_{j,\ell}(x):=|W^{1/p}(x)\vc_{j,\ell}|^p$ is doubling with a doubling constant $\beta>0$ independent of $j$ and $\ell$. We can therefore use Lemma \ref{le:sq} below to obtain the following estimate,
  \begin{align*}
\bigg\| \sum_\ell\vc_{j,\ell}\psi_{j,\ell}\bigg\|_{L^p(W)}^p&\leq 
C_N'r_j^{\nu p/2}\sum_\ell\int_{\bR^d}  |W^{1/p}(x)\vc_{j,\ell}|^p\, \big(1+r_j\big|x-x_{j,\ell}\big|_A\big)^{-\frac{Np}{2}}\,dx\\
&\leq 
C_N'r_j^{\nu p/2}\sum_{\ell\in\bZ^d}\int_{Q(j,\ell)}  |W^{1/p}(x)\vc_{j,\ell}|^p\, dx\\
&\asymp\sum_\ell \big\|  |Q(j,\ell)|^{-1/2} \vc_{j,\ell} \mathbf{1}_{Q(j,\ell)}\big\|_{L^p(W)}^p.
\end{align*}
We may now conclude that
\begin{align}
      \bigg\|\sum_{(j,\ell)\in F}\vc_{j,\ell}\psi_{j,\ell}\bigg\|_{M_{p,q}^{s}(W)}&\leq C
     \bigg\|\bigg\{  \sum_{j\in N_k}\bigg\|r_j^s \sum_\ell\vc_{j,\ell}\psi_{j,\ell}\bigg\|_{L^p(W)}\bigg\}_k\bigg\|_{\ell_q}\nonumber\\
     &\leq 
 C'  \bigg\|\bigg\{  \sum_{j\in N_k}r_j^s\bigg(\sum_\ell\big\|  |Q(j,\ell)|^{-1/2} \vc_{j,\ell} \mathbf{1}_{Q(j,\ell)}\big\|_{L^p(W)}^p\bigg)^{1/p}\bigg\}_k\bigg\|_{\ell_q}\nonumber\\  
   &\leq 
 C''  \bigg\|\bigg\{  \sum_{k}r_k^s\bigg(\sum_\ell \big\| |Q(k,\ell)|^{-1/2} \vc_{k,\ell} \mathbf{1}_{Q(k,\ell)}\big\|_{L^p(W)}^p\bigg)^{1/p}\bigg\}_k\bigg\|_{\ell_q}\nonumber\\
 &\lesssim \|\{\vc_{j,\ell}\}\|_{m_{p,q}^{s}(W)},\label{eq:LW}
 \end{align}
 where we have used the uniform bound on the cardinality of $N_k$ and
 that $r_j\asymp r_k$ (uniformly) for $j\in N_k$ by the moderation property of $h$. This concludes the proof.
 
\end{proof}

\begin{remark}\label{rem:dense}
Since $\{\psi_{k,\ell}\}_{k,\ell}\subset \mathcal{S}(\bR^d)$, it follows easily from Proposition  \ref{prop:coef} and Proposition \ref{prop:recon}, by standard arguments, that $\bigoplus_{j=1}^N\dS$ is dense in $M_{p,q}^{s}(W)$ whenever $0<p<\infty$, $0<q<\infty$, and $W\in\mathbf{A}_p(\cB_A)$.
\end{remark}

\begin{lemma}\label{le:sq}
Let $w:\bR^d \rightarrow (0,\infty)$ be a function satisfying the doubling condition
  \begin{equation*}\label{eq:doub2}
        \int_{B_A(x,2r)}w(t)\,dt\leq c\int_{B_A(x,r)} w(t)\,dt,\qquad x\in\bR^d,r>0,
    \end{equation*}
 with doubling exponent $\beta\geq \nu$ such that $2^\beta=c$. Let $j,\ell\in\bZ^d$ and let the quantities $Q(j,\ell)$, $r_j$, $x_{j,\ell}$ be defined by Eqs.\ \eqref{eq:qk}, \eqref{eq:rk} and \eqref{eq:xk}, respectively. Then for $L>\beta$, we have 
$$\int_{\bR^d}  w(x)\big(1+r_j\big|x-x_{j,\ell}\big|_A\big)^{-L}\,dx \leq C\int_{Q(j,\ell)}  w(x)\, \,dx,$$
where $C$ depends only on $c$ and $L$.
\end{lemma}
\begin{proof}
We make a partition $\bR^d=\cup_{m=0}^\infty R_m$, where $R_0=Q(j,\ell)$ and the annuli $R_m$, $m\geq 1$, are defined by
$$R_m:=\bigg\{y\in\bR^d: 2^{m-1}r_j^{-1}\leq |y-x_{j,\ell}|_A <  2^mr_j^{-1}\bigg\}.$$
Then 
\begin{align*}
 \int_{\bR^d}  w(x)\big(1+r_j\big|x-x_{j,\ell}\big|\big)^{-L}\,dx&=
\sum_{m=0}^\infty \int_{R_m}  w(x) \big(1+r_j\big|x-x_{j,\ell}\big|\big)^{-L}\,dx\\
&\leq C \sum_{m=0}^\infty 2^{-mL}\int_{R_m}  w(x)\, \,dx\\
\end{align*}
However, by the doubling property of $w(x)$, noting that $R_m\subseteq \{y: |y-x_{j,\ell}|_A < 2^mr_j^{-1}\}$, we have
$$\int_{R_m}  w(x)\, \,dx \leq c2^{\beta m}\int_{R_0}  w(x)\, \,dx,$$
so
$$\int_{\bR^d}  w(x)\big(1+r_j\big|x-x_{j,\ell}\big|_A\big)^{-L}\,dx \leq C' \sum_{m=0}^\infty  2^{(\beta -L)m}\int_{R_0}  w(x)\, \,dx\leq C''\int_{R_0}  w(x)\, \,dx,$$
provided  $L>\beta$.
\end{proof}

\section{Stable expansions and almost diagonal matrices}\label{sec:4}
The band-limited tight frame $\{\psi_{j,k}\}$ defined in \eqref{eq:ltf} provides a useful stable decomposition system for $M^{s}_{p,q}(W)$ whenever $W\in \mathbf{A}_p(\cB_A)$. It is, however, desirable to extend the stability results to cover more general systems of localized ``molecules" because this will allow us to study, e.g., boundedness of various operators acting on $M^{s}_{p,q}(W)$. In this section, we will study molecules in a discretized setting using an adapted notion of almost diagonal matrices. We first introduce a variation of $m_{p,q}^s(W)$ based on the notion of reducing operators.
\subsection{Reducing operators and the connection to scalar spaces}
It is known that for any matrix weight $W:\cD\rightarrow \bC^{N\times N}$, $0< p<\infty$, and $Q=Q(k,\ell)$, one of the cubes from \eqref{eq:qk}, there exists a nonnegative-definite matrix $A_Q$ such that for $\vx\in\bC^N$,
$$c_1|A_Q\vx|\leq  \rho_{p,Q}(\vx):=\left(\frac{1}{|Q|}\int_Q |W^{1/p}(t)\vx|^p\,dt\right)^{1/p}\leq c_2|A_Q\vx|,$$
where the equivalence constants $c_1$ and $c_2$ can be chosen independent of $Q$ and $\vx$. 
For $1\leq p<\infty$, the existence of $A_Q$ and the uniform equivalence constants follows directly from the Ellipsoid Theorem of F.\ John, see \cite{John:2014a}, while for $0<p<1$ the argument is more involved and can be found in \cite[Section 5]{FrazRoud:2004a}. 
$A_Q$ is referred to as a \textit{reducing operator} for the norm $\rho_{p,Q}$ on $\bC^N$. Frazier and Roudenko, see \cite{Roud:2003a,FrazRoud:2004a,FrazRoud:2021a}, made the important observation that reducing operators can be used to make certain connections between the matrix-weighted Besov space and the scalar $\phi$-transform introduced by Frazier and Jawerth \cite{FrazJawe:1990a}. We will now use a similar approach to study the vector-valued modulation-type spaces. We shall need the following definition.

\begin{definition}\label{def:red}
Let $W:\Rd\rightarrow \bC^{N\times N}$ be a matrix weight, $s\in \bR$, $0< p<\infty$, and $0<q\leq \infty$. Suppose $\cC$ is an admissible covering of $\bR^d$ with associated BAPU 
   $\{\phi_k\}_{k\in\bZ^d}$ of the type considered in Definition \ref{def:bapu}. Let $\{A_{Q(k,\ell)}\}_{k,\ell\in\bZ^d}$ be a family of reducing operators associated with $W$. For any vector-valued sequence $\{\vs_{k,\ell}\}_{k,\ell}$, we define, 
$$\|\{\vs_{k,\ell}\}_{k,\ell}\|_{m_{p,q}^{s}(\{A_{Q(k,\ell)}\})}:=
\bigg\|\bigg\{ r_k^s\bigg(\sum_{\ell\in\bZ^d} \big\| |Q(k,\ell)|^{-\frac{1}{2}}|A_{Q(k,\ell)}\vs_{k,\ell}|\mathbf{1}_{Q(k,\ell)}\big\|_{L^p(dt)}^p\bigg)^{1/p}\bigg\}_k\bigg\|_{\ell^q},$$
and
$$m_{p,q}^{s}(\{A_{Q(k,\ell)}\})=\Big\{\{\vs_{k,\ell}\}_{k,\ell}:\|\{\vs_{k,\ell}\}_{k,\ell}\|_{m_{p,q}^{s}(\{A_{Q(k,\ell)}\})}<\infty\Big\}.$$
\end{definition}

We have the following observation.
\begin{lemma}\label{le:eqi}
Let $W:\Rd\rightarrow \bC^{N\times N}$ be a matrix weight, and let $\cC$ be an admissible covering of $\bR^d$ with associated BAPU 
   $\{\phi_k\}_{k\in\bZ^d}$ of the type considered in Definition \ref{def:bapu}. Let $s\in \bR$, $0< p<\infty$, and $0<q\leq \infty$. For any finite vector-valued sequence $\vs=\{\vs_{k,\ell}\}_{k,\ell}$, we have
$$\|\{\vs_{k,\ell}\}_{k,\ell}\|_{m_{p,q}^{s}(W)}\asymp
\|\{\vs_{k,\ell}\}_{k,\ell}\|_{m_{p,q}^{s}(\{A_{Q(k,\ell)}\})},$$
with equivalence constants independent of $\vs$.
\end{lemma}
\begin{proof}
We simply adapt the proof presented in \cite{FrazRoud:2004a} for dyadic cubes to the current setup.
\begin{align*}
    \|\{\vs_{k,\ell}\}_{k,\ell}\|_{m_{p,q}^{s}(W)}
    &=
    \bigg\|\bigg\{r_k^s \bigg(\sum_{\ell\in\bZ^d} \big\||Q(k,\ell)|^{-\frac{1}{2}}\big|W^{1/p}(t)\vs_{k,\ell}\big|\mathbf{1}_{Q(k,\ell)}(t)\big\|_{L^p(dt)}^p\bigg)^{1/p}\bigg\}_k\bigg\|_{\ell^q}\\
    &\asymp \bigg\|\bigg\{ r_k^s\bigg(\sum_{\ell\in\bZ^d} |Q(k,\ell)|^{-\frac{p}{2}}[\rho_{p,Q(k,\ell)}(\vs_{k,\ell})^p|Q(k,\ell)|\bigg\}_k\bigg\|_{\ell^q}\\
     &\asymp \bigg\|\bigg\{ r_k^s\sum_{\ell\in\bZ^d} |Q(k,\ell)|^{-\frac{p}{2}}|A_{Q(k,\ell)}\vs_{k,\ell}|^p|Q(k,\ell)|\bigg\}_k\bigg\|_{\ell^q}\\
     &\asymp\bigg\|\bigg\{ r_k^s\bigg(\sum_{\ell\in\bZ^d}  \big\| |Q(k,\ell)|^{-\frac{1}{2}}|A_{Q(k,\ell)}\vs_{k,\ell}|\mathbf{1}_{Q(k,\ell)}\big\|_{L^p(dt)}^p\bigg)^{1/p}\bigg\}_k\bigg\|_{\ell^q}\\
     &=\|\{\vs_{k,\ell}\}_{k,\ell}\|_{m_{p,q}^{s}(\{A_{Q(k,\ell)}\})}.
\end{align*}
\end{proof}

There is a straightforward connection between the (quasi-)norm given by $\|\,\cdot\,\|_{m_{p,q}^{s}(\{A_{Q(k,\ell)}\})}$ and the scalar discrete decomposition space norm $\|\cdot\|_{m^{s}_{p,q}}$, defined for a scalar sequence $t:=\{t_{k,\ell}\}$ by
\begin{equation}\label{eq:scalar}
    \|t\|_{m^{s}_{p,q}}:=\bigg\|\bigg\{ r_k^s\bigg(\sum_{\ell\in\bZ^d}\big\| |Q(k,\ell)|^{-\frac{1}{2}}t_{k,\ell}\mathbf{1}_{Q(k,\ell)}\big\|_{L^p(dt)}^p\bigg)^{1/p}\bigg\}_k\bigg\|_{\ell^q},
\end{equation}
see \cite{BoruNiel:2008a,NielRasm:2012a} for further details on discrete (scalar) decomposition spaces.
Given a vector-valued sequence $\vs:=\{\vs_{k,\ell}\}_{k,\ell}$, we define the scalar sequence $t:=\{t_{Q(k,\ell)}\}$ by putting
$$t_{k,\ell}:=|A_{Q(k,\ell)}\vs_{k,\ell}|.$$
Then we clearly have
\begin{equation}\label{eq:connect}
    \|\vs\|_{m_{p,q}^{s}(\{A_{Q(k,\ell)}\})}=\|t\|_{m^{s}_{p,q}}.
\end{equation}
The direct connection to scalar decomposition spaces provided by Eq.\ \eqref{eq:connect} simplifies the study of matrix-valued function spaces, as several results from the existing theory of scalar decomposition spaces can simply be restated in the matrix-weighted setting. This fundamental connection was first identified by Frazier and Roudenko in their work on matrix-weighted Besov spaces \cite{FrazRoud:2004a}.

\subsection{Almost diagonal matrices} The identity provided by Eq.\ \eqref{eq:connect}, combined with Lemma \ref{le:eqi}, allows us to use operators on the scalar space $m^{s}_{p,q}$ to study the vector-valued $m^{s}_{p,q}(W)$.  For this purpose, it will be useful to recall the notion of almost diagonal matrices for the scalar spaces $m^{s}_{p,q}$ introduced by Rasmussen and the author in \cite{NielRasm:2012a}.

In order to maintain ``compatibility'' with the results in \cite{NielRasm:2012a}, we add some further restrictions to the moderate function $h$ used to generate admissible coverings.
We shall henceforth assume that there exist $\kappa,R_1,\rho_1>0$ such that:
\begin{equation}\label{eq:cases}
 \begin{cases}
    h^{1+\kappa}\text{ is moderate, and,}\\
     |\xi-\zeta|_A\leq ah(\xi)\text{ for }a\geq \rho_1 \text{ implies } h(\zeta)\leq R_1ah(\xi).
    \end{cases}
\end{equation}
It should be noted that the additional assumptions are not excessively restrictive. As an example, one can easily verify that
$h_\eta(\xi)=(1+|\xi|_A)^\eta$, for $0< \eta<1$, satisfies all the requirements. 

With these assumptions in place, we can now define various classes of almost diagonal matrices.
\begin{definition}\label{doitmad}%
Let $\cC=\{C_k\}_{k\in\bZ^d}$ be an admissible covering of $\bR^d$ induced by the affine transformations $T_k y:=\delta_{r_k}y+\xi_k$.
Assume that $s\in \bR$, $0<q< \infty$, and $0<{p}<\infty$. A matrix $\mathbf{B}:=\{b_{(j,m)(k,n)}\}_{j,m,k,\ell\in\bZ^d}$
is called almost diagonal on $m^{s}_{{p},q}$ if
there exist $J\geq \frac{\nu}{\min\{1,p,q\}}$ and  $C, \delta>0$ such that
\begin{align*}
|b_{(j,\ell)(k,m)}|
&\le C \omega_{(j,\ell)(k,m)}^s(J), \qquad j,m,k,n\in \Zn,
\end{align*}
where
\begin{align*}
\omega_{(j,\ell)(k,m)}^s(J):=\bigg(\frac{r_k}{r_j}\bigg)^{s+\frac{\nu}{2}}
\min\bigg(\bigg(\frac{r_j}{r_k}\bigg)^{J+\frac{\delta}{2}},&\bigg(\frac{r_k}{r_j}\bigg)^{\frac{\delta}{2}}\bigg) c_{jk}^\delta(J)\\
&\phantom{C}\times
(1+\min(r_k,r_j)|x_{k,m}-x_{j,\ell}|_A)^{-J-\delta},
\end{align*}
with
\begin{equation*}
c_{jk}^\delta(J):=\min\bigg(\bigg(\frac{r_j}{r_k}\bigg)^{J+\delta},\bigg(\frac{r_k}{r_j}\bigg)^{\delta}\bigg)
(1+\max(r_k,r_j)^{-1}|\xi_k-\xi_j|_A)^{-J-\delta}
\end{equation*}
with $r_k$, $\xi_k$, and $x_{k,n}$ defined in Eqs.\ \eqref{eq:rk} and \eqref{eq:xk}.
We denote the set of almost diagonal matrices on $m^{s}_{{p},q}$  by $\textrm{ad}_{{p},q}^{s}$.
\end{definition}
\begin{remark}\label{rem:ad}
A more symmetric sufficient condition for the matrix $\mathbf{B}:=\{b_{(j,\ell)(k,m)}\}_{j,m,k,\ell\in\bZ^d}$ to be in $\textrm{ad}_{{p},q}^{s}$ can be obtained by requiring the existence of  $J>\frac{\nu}{\min\{1,p,q\}}$ and $M>\min\{2J,|s|+\nu/2\}$ such that
\begin{align}
    |b_{(j,\ell)(k,m)}|&\leq  C\min\bigg\{\bigg(\frac{r_j}{r_k}\bigg)^{M},\bigg(\frac{r_k}{r_j}\bigg)^{M}\bigg\}(1+\min(r_k,r_j)|x_{k,m}-x_{j,\ell}|_A)^{-J}\nonumber\\
&\hspace{2cm}\times (1+\max(r_k,r_j)^{-1}|\xi_k-\xi_j|_A)^{-J}.\label{eq:ad_symm}
\end{align}   
\end{remark}
Any matrix $\mathbf{B}=\{b_{(j,\ell)(k,m)}\}_{j,m,k,\ell\in\bZ^d}$ in $\textrm{ad}_{{p},q}^{s}$ induces a linear operator on $m^{s}_{p,q}$ by the action of the standard matrix vector product. Specifically, let $s\in m^{s}_{p,q}$ be a finite sequence and put $t=\mathbf{B}s$, i.e.,
$$t_{j,\ell}:=\sum_{(k,m)\in\bZ^d\times\bZ^d} b_{(j,\ell)(k,m)}s_{k,m},\qquad (j,\ell)\in \bZ^d\times\bZ^d.$$
It was proved in \cite{NielRasm:2012a} that almost diagonal matrices act boundedly on the class $m^{s}_{p,q}$.

\begin{proposition}\label{johnlennon}%
Suppose that $\mathbf{B} \in \textrm{ad}_{{p},q}^{s}$. Then $\mathbf{B}$ is bounded on
$m^{s}_{p,q}$.
\end{proposition}

Next, we observe that the following elementary matrix estimate holds 
$$|A_{Q(k,\ell)}\vc|=|A_{Q(k,\ell)}A_{Q(j,m)}^{-1}A_{Q(j,m)}\vc|\leq \|A_{Q(k,\ell)}A_{Q(j,m)}^{-1}\|\cdot |A_{Q(j,m)}\vc|.$$
It is therefore of interest to study families of reducing operators that are ``compatible'' with the almost diagonal classes introduced. This leads to the following definition.
\begin{definition}
 Let $\cQ$ be the collection of sets defined in \eqref{eq:Q}, let $\{A_Q\}_{Q\in \cQ}$ be a sequence of nonnegative-definite matrices, and let $\beta>0$, $0<p<\infty$. We say that $\{A_Q\}_{Q\in \cQ}$ is strongly doubling of order $(\beta,p)$ if there exists $c>0$ such that for $P=Q(k,m)$ and $Q=Q(j,\ell)$,
\begin{equation}\label{eq:AQ}
    \|A_QA_P^{-1}\|\leq c\max\bigg\{\left(\frac{r_j}{r_k}\right)^{\nu/p},\left(\frac{r_k}{r_j}\right)^{(\beta-\nu)/p}\bigg\}
\big(1+\min\{r_j,r_k\}|x_{j,\ell}-x_{k,m}|_A\big)^{\beta/p}.\end{equation}
\end{definition}

We have the following lemma, where we recall that, as an important special case, any $W\in \mathbf{A}_p(\cB_A)$ is doubling of order $p$.
\begin{lemma}\label{le:dou}
Let $W$ be a doubling matrix weight of order $p > 0$ with doubling exponent $\beta$ as specified in Definition \ref{def:doub} and suppose $\{ A_Q\}_{Q\in\cQ}$ is a sequence of reducing operators of order $p$ for $W$. Then $\{A_Q\}$ is strongly doubling of order $(\beta,p)$.
\end{lemma}
\begin{proof}
Let $Q=Q(j,\ell), P=Q(k,m)\in\cQ$, and let $\omega\geq 1$ be a minimal constant for which
$Q\subset B_A(x_{k,m},\omega r_k^{-1})=:\delta_\omega P$. 
Let us find an upper bound for $\omega$. Pick any $z\in Q(j,\ell)$, and notice by the quasi-triangle inequality \eqref{eq:Qtriangle},
\begin{equation}\label{eq:Qa}
|x_{k,m}-z|_A\leq C_A(|x_{j,\ell}-z|_A+|x_{k,m}-x_{j,\ell}|_A)\leq C_A(r_j^{-1}+|x_{k,m}-x_{j,\ell}|_A).
\end{equation}
Hence, if $r_k\leq r_j$, we simply use $r_j^{-1}\leq r_k^{-1}$ in \eqref{eq:Qa} to obtain
$$|x_{k,m}-z|_A \leq C_A(r_k^{-1}+|x_{k,m}-x_{j,\ell}|_A)=C_Ar_k^{-1}(1+r_k|x_{k,m}-x_{j,\ell}|_A).$$
In the other case, $r_j< r_k$, we obtain directly from \eqref{eq:Qa} that
$$|z-x_{k,m}|_A \leq C_A{r_k^{-1}}\bigg(\frac{r_k}{r_j}+r_k|x_{k,m}-x_{j,\ell}|_A\bigg).$$
Hence, we obtain the general estimate
\begin{equation}\label{eq:alp}
\omega\leq C_A \max\bigg\{1,\frac{r_k}{r_j}\bigg\}\big(1+\min\{r_j,r_k\}|x_{j,\ell}-x_{k,m}|_A\big).
\end{equation}
By the doubling property for $w_\vy(\cdot):=|W^{1/p}(\cdot)\vy|^p$, see Definition \ref{def:doub}, we have
$$w_\vy(Q)\leq w_\vy(\omega P)\leq c\omega^\beta w_\vy(P).$$
Hence,
$$|A_Q\vy|^p\lesssim \frac{1}{|Q|}\int_Q |W^{1/p}(x)\vy|^p\, dx= \frac{1}{|Q|} w_\vy(Q)
\lesssim \frac{1}{|Q|}\omega^\beta w_\vy(P)\lesssim
\frac{|P|}{|Q|}\omega^\beta |A_P\vy|^p.$$
Now we put $\vy=A_P^{-1}\vz$ for arbitrary $\vz\in\bR^d$, and recall that $|P|\asymp r_k^{-\nu}$, $|Q|\asymp r_j^{-\nu}$. Using estimate \eqref{eq:alp}, we obtain
$$|A_QA_P^{-1}\vz|^p\leq c \left(\frac{r_j}{r_k}\right)^\nu \max\bigg\{1,\frac{r_k}{r_j}\bigg\}^\beta
\big(1+\min\{r_j,r_k\}|x_{j,\ell}-x_{k,m}|_A\big)^\beta|\vz|^p,$$
and \eqref{eq:AQ} follows.
\end{proof}

The preceding lemma is deliberately formulated under the elementary doubling assumption.  For weights in the matrix Muckenhoupt class one can sometimes use a sharper form of the estimate, based on the $A_p$-dimension of the weight in the sense of Bu--Hyt\"onen--Yang--Yuan \cite[Lemma~2.29]{BuHytoYang:2024a}.  We record the anisotropic version here as a complement to Lemma \ref{le:dou}.

\begin{definition}\label{def:anisotropic-Ap-dimension}
Let $0<p<\infty$, let $d_0\in [0,\nu)$, and let $W\in\mathbf{A}_p(\cB_A)$. We say that $W$ has anisotropic $A_p$-dimension $d_0$ if there exists a constant $C>0$ such that, for every ball $B=B_A(x,r)\in\cB_A$ and every $\lambda\geq 1$, the following estimates hold.

If $0<p\leq 1$, then
\begin{equation}\label{eq:anisotropic-Ap-dim-small-p}
 \operatorname*{ess\,sup}_{y\in \lambda B}
 \fint_B \|W^{1/p}(z)W^{-1/p}(y)\|^p\,dz
 \leq C\lambda^{d_0}.
\end{equation}
If $1<p<\infty$, then
\begin{equation}\label{eq:anisotropic-Ap-dim-large-p}
 \fint_B
 \left(
   \fint_{\lambda B}
   \|W^{1/p}(z)W^{-1/p}(y)\|^{p'}\,dy
 \right)^{p/p'}dz
 \leq C\lambda^{d_0},
\end{equation}
where $p'$ denotes the conjugate exponent of $p$. Here
$\lambda B:=B_A(x,\lambda r)$.
\end{definition}

\begin{definition}\label{def:anisotropic-dimensions-triple}
Let $W\in\mathbf{A}_p(\cB_A)$. We say that $W$ has anisotropic dimension data $(d,\widetilde d)$ if $W$ has anisotropic $A_p$-dimension $d\in [0,\nu)$ and, in addition, either $0<p\leq 1$ and $\widetilde d=0$, or $1<p<\infty$ and the dual weight
\[
   \widetilde W:=W^{-1/(p-1)}
\]
has anisotropic $A_{p'}$-dimension $\widetilde d$. We put
\begin{equation}\label{eq:dimension-losses}
   \sigma_p(W):=
   \begin{cases}
      0, & 0<p\leq 1,\\[2mm]
      \widetilde d/p', & 1<p<\infty,
   \end{cases}
   \qquad
   \Delta_p(W):=d/p+\sigma_p(W),
   \qquad
   K_p(W):=\max\{d/p,\sigma_p(W)\}.
\end{equation}
\end{definition}

\begin{lemma}\label{le:dou-dimension}
Let $0<p<\infty$, and let $W\in\mathbf{A}_p(\cB_A)$ have anisotropic dimension data $(d,\widetilde d)$ in the sense of Definition \ref{def:anisotropic-dimensions-triple}. Let $\{A_Q\}_{Q\in\cQ}$ be a family of reducing operators of order $p$ associated with $W$. Then, for $Q=Q(j,\ell)$ and $P=Q(k,m)$,
\begin{equation}\label{eq:AQ-dimension}
 \|A_QA_P^{-1}\|
 \le C
 \max\left\{
 \left(\frac{r_j}{r_k}\right)^{d/p},
 \left(\frac{r_k}{r_j}\right)^{\sigma_p(W)}
 \right\}
 \left(
 1+\min\{r_j,r_k\}|x_{j,\ell}-x_{k,m}|_A
 \right)^{\Delta_p(W)},
\end{equation}
where $\sigma_p(W)$ and $\Delta_p(W)$ are defined in \eqref{eq:dimension-losses}. The constant is independent of $j,k,\ell,m$.
\end{lemma}

\begin{proof}
We only indicate the adaptation, since the argument is the same as the proof of \cite[Lemma~2.29]{BuHytoYang:2024a}, with Euclidean cubes replaced by anisotropic balls and with the homogeneous dimension $\nu$ replacing the Euclidean dimension.  The reducing-operator equivalences give, for nested anisotropic balls $B\subset \lambda S$, the estimate
\[
   \|A_BA_S^{-1}\|\lesssim \lambda^{d/p},
\]
while, for $1<p<\infty$, applying the same argument to the dual weight $\widetilde W=W^{-1/(p-1)}$ gives the reverse scale estimate with exponent $\widetilde d/p'$.  For $0<p\leq 1$ there is no dual contribution, in accordance with the definition of $\sigma_p(W)$.

Now let $Q=Q(j,\ell)$ and $P=Q(k,m)$. Choose an anisotropic ball $S$ such that $Q\cup P\subset S$ and
\[
 \operatorname{rad}(S)
 \lesssim r_j^{-1}+r_k^{-1}+|x_{j,\ell}-x_{k,m}|_A .
\]
Then
\[
 \|A_QA_P^{-1}\|
 \leq \|A_QA_S^{-1}\|\,\|A_SA_P^{-1}\|.
\]
Applying the nested estimates to $Q\subset S$ and $P\subset S$ yields
\[
 \|A_QA_P^{-1}\|
 \lesssim
 \left(\frac{\operatorname{rad}(S)}{r_j^{-1}}\right)^{d/p}
 \left(\frac{\operatorname{rad}(S)}{r_k^{-1}}\right)^{\sigma_p(W)} .
\]
Finally,
\[
 \operatorname{rad}(S)
 \lesssim
 \max\{r_j^{-1},r_k^{-1}\}
 \left(1+\min\{r_j,r_k\}|x_{j,\ell}-x_{k,m}|_A\right),
\]
and this gives \eqref{eq:AQ-dimension}.
\end{proof}

\begin{remark}\label{rem:dimension-almost-diagonal}
Lemma \ref{le:dou-dimension} may be used in place of Lemma \ref{le:dou} whenever the anisotropic dimension data are available.  In the definition of the matrix-weighted almost diagonal class below, this amounts to replacing the doubling loss $K=\frac{\beta}{p}$
 by the sharper losses
\[
   K_p(W)=\max\{d/p,\sigma_p(W)\},
   \qquad \Delta_p(W)=d/p+\sigma_p(W).
\]
Thus the spatial decay exponent $J+\beta/p$ may be replaced by
$J+\Delta_p(W)$, and the scale-loss parameter $K$ may be replaced by
$K_p(W)$.  The original formulation is retained below because it only requires the elementary doubling hypothesis.
\end{remark}

We can now define the class of almost diagonal matrices adapted to the vector-valued sequence space $m^{s}_{{p},q}(W)$. According to Remark \ref{rem:doub}, the definition in particular applies to the setup where the matrix weight is in $\mathbf{A}_p(\cB_A)$.  
\begin{definition}\label{def:mad}
Let $W$ be a doubling matrix weight of order $p > 0$ with doubling exponent $\beta$ as specified in Definition \ref{def:doub}. Put $K:=\frac{\beta}{p}$. A matrix $\mathbf{B}:=\{b_{(j,\ell)(k,m)}\}_{j,m,k,\ell\in\bZ^d}$
is called almost diagonal on $m^{s}_{{p},q}(W)$ if
there exist $J>\frac{\nu}{\min\{1,p,q\}}$, $M>\max\{2J,|s|+\nu/2\}$,  and $C>0$ such that
\begin{align}
    |b_{(j,\ell)(k,m)}|&\leq  C\min\bigg\{\bigg(\frac{r_j}{r_k}\bigg)^{M+K},\bigg(\frac{r_k}{r_j}\bigg)^{M+K}\bigg\}(1+\min(r_k,r_j)|x_{k,m}-x_{j,\ell}|_A)^{-J-\frac{\beta}{p}}\nonumber\\
&\hspace{2cm}\times (1+\max(r_k,r_j)^{-1}|\xi_k-\xi_j|_A)^{-J},\label{eq:mad_symm}
\end{align} 
with $r_k$, $\xi_k$, and $x_{k,n}$ defined in Eqs.\ \eqref{eq:rk} and \eqref{eq:xk}.
We denote the set of almost diagonal matrices on $m^{s}_{{p},q}$(W) by $\textbf{ad}_{{p},q}^{s}(W)$.

\end{definition}
The following result provides a fundamental boundedness result for matrices in $\textbf{ad}_{{p},q}^{s}$ on the vector-valued sequence space $m^{s}_{{p},q}(W)$. 
\begin{proposition}\label{prop:di}
Let $s\in \bR$, $0< p<\infty$, $0<q< \infty$, and let $W\in\mathbf{A}_p(\cB_A)$
with order $p$ doubling exponent $\beta$ as specified in Definition \ref{def:doub}. Let $\{A_Q\}_{Q\in\cD}$ be a family of reducing operators associated with $W$. Suppose $\mathbf{B}:=\{b_{(j,m)(k,n)}\}_{j,m,k,\ell\in\bZ^d}$ is almost diagonal with $\mathbf{B}\in \textbf{ad}^{s}_{p,q}(W)$. Then $\mathbf{B}$ is bounded on $m^{s}_{p,q}(W)$.
\end{proposition}
\begin{proof}
Let $\vs\in m^{s}_{p,q}(W)$ be a finite sequence and put $\vt=\mathbf{B}\vs$, i.e.,
$$\vt_{j,\ell}:=\sum_{(k,m)\in\bZ^d\times\bZ^d} b_{(j,\ell)(k,m)}\vs_{k,m},\qquad (j,\ell)\in \bZ^d\times\bZ^d.$$
There are no convergence issues due to the fact that $\vs$ is finite.
Now we define an associated scalar sequence $t=(t_{j,k})_{j,k\in\bZ^d}$ by letting 
$t_{j,\ell}=|A_{Q(j,\ell)}\vt_{j,\ell}|$. Then we notice that
\begin{align*}
    t_{j,\ell}&=|A_{Q(j,\ell)}\vt_{j,\ell}|\\
    &=\bigg|A_{Q(j,\ell)}\sum_{(k,m)\in\bZ^d\times\bZ^d} b_{(j,\ell)(k,m)}\vs_{k,m}\bigg|\\
    &\leq \sum_{(k,m)\in\bZ^d\times\bZ^d} |b_{(j,\ell)(k,m)}|\cdot|A_{Q(j,\ell)}\vs_{k,m}|\\
    &\leq  \sum_{(k,m)\in\bZ^d\times\bZ^d} |b_{(j,\ell)(k,m)}|\cdot\|A_{Q(j,\ell)}A_{Q(k,m)}^{-1}\||A_{Q(k,m)}\vs_{k,m}|\\
    &=\sum_{(k,m)\in\bZ^d\times\bZ^d} \gamma_{(j,\ell)(k,m)} s_{k,m}, 
\end{align*}
with $s_{k,m}:=|A_{Q(k,m)}\vs_{k,m}|$ and $\gamma_{(j,\ell)(k,m)}:=|b_{(j,\ell)(k,m)}|\|A_{Q(j,\ell)}A_{Q(k,m)}^{-1}\|$. Hence, using the observation in Lemma \ref{le:eqi}, we have
$$\|\{t_{j,m}\}_{j,m}\|_{m^{s}_{p,q}}=\|\{\vt_{j,m}\}_{(j,m)}\|_{m^{s}_{p,q}(\{A_Q\})}\asymp \|\{\vt_{j,m}\}_{j,m}\|_{m^{s}_{p,q}(W)},$$
and 
$$\|\{s_{j,m}\}_{j,m}\|_{m^{s}_{p,q}}=\|\{\vs_{j,m}\}_{j,m}\|_{m^{s}_{p,q}(\{A_Q\})}\asymp \|\{\vs_{j,m}\}_{j,m}\|_{m^{s}_{p,q}(W)}.$$
 Therefore, to prove the desired boundedness result, it suffices to verify that
$\Gamma:=(\gamma_{(j,m)(k,\ell)})\in\mathrm{ad}_{p,q}^{s}$ since, in the scalar setting, an almost diagonal matrix for $m^{s}_{p,q}$ will map $m^{s}_{p,q}$ boundedly into $m^{s}_{p,q}$ according to Proposition \ref{johnlennon}. We notice that the estimate by Lemma \ref{le:dou}, and the almost diagonal assumption on $\mathbf{B}$ given by \eqref{eq:mad_symm}, ensure that
there exist $J>\frac{\nu}{\min\{1,p,q\}}$ and $M>\max\{2J,|s|+\nu/2\}$ such that
\begin{align*}
    |\gamma_{(j,m)(k,n)}|&\leq  C\min\bigg\{\bigg(\frac{r_j}{r_k}\bigg)^{M},\bigg(\frac{r_k}{r_j}\bigg)^{M}\bigg\}(1+\min(r_k,r_j)|x_{k,n}-x_{j,m}|_A)^{-J}\nonumber\\
&\hspace{2cm}\times (1+\max(r_k,r_j)^{-1}|\xi_k-\xi_j|_A)^{-J},
\end{align*}   
and, as noted in Remark \ref{rem:ad}, this implies that $\Gamma\in\mathrm{ad}_{p,q}^{s}$. 
\end{proof}

We now turn to a first useful application of Proposition \ref{prop:di}: the study of ``change of frame'' operators. As we will see in Corollaries  \ref{cor:recon} and \ref{cor:coeff} below, we can use the ``change of frame'' operators to extend Propositions \ref{prop:coef} and \ref{prop:recon} to cover much more general expansion systems.

 Let  $\kappa>0$ be the constant defined in Eq.\ \eqref{eq:cases} and let
$\{\psi_{k,n}\}_{k,n\in \bZ^d}$ be the tight frame defined in
\eqref{eq:ltf} for the given decomposition $\cC$. It can easily be verified that for any fixed $N,P,L>0$, $\psi_{k,n}$
satisfies the following decay estimates in the spatial and frequency domains,
\begin{align}
&|\psi_{k,m}(x)|\le Cr_k^{\frac{\nu}{2}}(1+r_k|x_{k,m}-x|_A)^{-2N}\label{one},\\
&|\hat{\psi}_{k,m}(\xi)|\le C
r_k^{-\frac{\nu}{2}}(1+r_k^{-1}|\xi_{k}-\xi|_A)^{-2L-2\kappa P},\label{three}
\end{align}
where $C$ is independent of $k$ and $m$, and as before,
\begin{equation}\label{dotodotodoto}
x_{k,m}=\delta_{r_k^{-1}}m,\,\, k,m\in\bZ^d,
\end{equation}
with $r_k$  defined in \eqref{eq:rk}. Let
$\{\eta_{k,n}\}_{k,n\in \bZ^d}\subset L^2(\R^d)$ be another system with similar decay properties,
\begin{align}
&|\eta_{j,m}(x)|\le Cr_j^{\frac{\nu}{2}}(1+r_j|x_{j,m}-x|_A)^{-2N}\label{two},\\
&|\hat{\eta}_{j,m}(\xi)|\le C
r_j^{-\frac{\nu}{2}}(1+r_j^{-1}|\xi_{j}-\xi|_A)^{-2L-2\kappa{P}}.\label{four}
\end{align}

The following lemma was proved in \cite{NielRasm:2012a}.
\begin{lemma}\label{bubbleboy}%
 Choose parameters $N,P,L>0$ such that $2N>\nu$ and $2L+2\kappa\frac{P-\nu}2>\nu$. If both systems $\{\eta_{k,n}\}_{k,n\in\bZ^d}$ and $\{\psi_{j,m}\}_{j,m\in\bZ^d}$ satisfy \eqref{two} and \eqref{four}, then we have
\begin{align*}
|\langle\eta_{k,n},\psi_{j,m}\rangle|
\le& C
\min\bigg(\frac{r_k}{r_j},\frac{r_j}{r_k}\bigg)^{P}(1+\max(r_k,r_j)^{-1}|\xi_k-\xi_j|_A)^{-L}\\
&\phantom{C} \times(1+\min(r_k,r_j)|x_{k,n}-x_{j,m}|_A)^{-N}.
\end{align*}
\end{lemma}

The lemma can be applied to obtain the following result yielding a reconstruction bound for any system of "molecules" with the same general structure as the frame $\{\phi_{k,n}\}_{k,n\in \bZ^d}$. The result in Corollary \ref{cor:recon} can be used to introduce a natural reconstruction/synthesis operator $T$ for the system of molecules, where the comment in Remark \ref{rem:T} applies.

\begin{corollary}\label{cor:recon}
 Let  $\kappa>0$ be the constant defined in Eq.\ \eqref{eq:cases}. Suppose  $s\in \bR$, $0< p<\infty$, and $0<q< \infty$, and let $W\in\mathbf{A}_p(\cB_A)$
with order $p$ doubling exponent $\beta$ as specified in Definition \ref{def:doub}.
Put $K:=\frac{\beta}{p}$ and choose parameters $N,P,L>0$ such that 
$2N>\nu$ and $2L+2\kappa\frac{P-\nu}2>\nu$, and additionally
\begin{align*}
     \min\{L,N\}>\frac{\nu}{\min\{1,p,q\}}+\frac{\beta}{p},\qquad P>K+\max\bigg\{
     \frac{2\nu}{\min\{1,p,q\}},|s|+\frac{\nu}{2}\bigg\}.
\end{align*} 
  If the system $\{\eta_{j,m}\}_{j,m\in\bZ^d}\subset L_2(\bR^d)$ satisfies \eqref{two} and \eqref{four} with parameters $N,P,L$ as specified, then there exists a constant $C$ such that for any finite vector-valued coefficient sequence  $\vs:=\{\vc_{j,\ell}\}_{(j,\ell)\in F}$, $F\subset \bZ^d\times\bZ^d$,
  \begin{equation}
      \bigg\|\sum_{(j,\ell)\in F}\vc_{j,\ell}\eta_{j,\ell}\bigg\|_{M_{p,q}^{s}(W)}\leq C \|\{\vc_{j,\ell}\}\|_{m_{p,q}^{s}(W)}.
  \end{equation}
\end{corollary}
\begin{proof}
We expand $\vf:=\sum_{(j,\ell)\in F}\vc_{j,\ell}\eta_{j,\ell}$ in the canonical system $\{\psi_{j,\ell}\}$. This yields
$$\vf=\sum_{(k,m)\in \bZ^d\times \bZ^d} (\mathbf{B}\vs)_{(j,m)}\psi_{k,m},$$
with $$\mathbf{B}=(\langle \eta_{j,\ell},\psi_{k,n}\rangle)_{(j,\ell)(k,n)}.$$
By Proposition \ref{prop:di}, $\|\mathbf{B}\vs\|_{m^{s}_{p,q}(W)}\leq C_1\|\vs\|_{m^{s}_{p,q}(W)}$, and it follows by Proposition \ref{prop:recon} that
$$\|\vf\|_{M^{s}_{p,q}(W)}\leq C_2\|\mathbf{B}\vs\|_{m^{s}_{p,q}(W)}\leq C_1C_2\|\vs\|_{m^{s}_{p,q}(W)}.$$
\end{proof}

 Using a similar type of argument, we can also obtain an estimate for the analysis/coefficient operator for any system with the same general structure as the frame $\{\phi_{k,n}\}_{k,n\in \bZ^d}$.
\begin{corollary}\label{cor:coeff}
Let  $\kappa>0$ be the constant defined in Eq.\ \eqref{eq:cases}, $s\in \bR$, $0< p<\infty$, and $0<q< \infty$, and let $W\in\mathbf{A}_p(\cB_A)$
with order $p$ doubling exponent $\beta$ as specified in Definition \ref{def:doub}.
Put $K:=\frac{\beta}{p}$ and choose $N,P,L>0$ such that 
$2N>\nu$ and $2L+2\kappa\frac{P-\nu}2>\nu$, and additionally
\begin{align*}
     \min\{L,N\}>\frac{\nu}{\min\{1,p,q\}}+\frac{\beta}{p},\qquad P>K+\max\bigg\{
     \frac{2\nu}{\min\{1,p,q\}},|s|+\frac{\nu}{2}\bigg\}.
\end{align*} 
  If the system $\{\eta_{j,m}\}_{j,m\in\bZ^d}\subset L^2(\R^d)$ satisfies \eqref{two} and \eqref{four} with parameters $N,P,L$ as specified, then we have 
  for $\vf\in M_{p,q}^{s}(W)$,
  \begin{equation}
      \|\{\vc_{k,\ell}\}_{k,\ell}\|_{m_{p,q}^{s}(W)}\leq C \|\vf\|_{M_{p,q}^{s}(W)},
  \end{equation}
  with $\vc_{k,\ell}:=\langle \vf,\eta_{k,\ell}\rangle$.
\end{corollary}
\begin{proof}
 By Proposition \ref{prop:coef}, there exists $C_1$ such that for  $\vf\in M_{p,q}^{s}(W)$,
 $$ \|\{\vs_{k,\ell}\}_{k,\ell}\|_{m_{p,q}^{s}(W)}\leq C_1 \|\vf\|_{M_{p,q}^{s}(W)},$$
 with $\vs_{k,\ell}:=\langle \vf,\psi_{k,\ell}\rangle$. We also notice that 
 $$\vf=\sum_{(j,m)\in\bZ^d\times\bZ^d}\vs_{j,m}\psi_{j,m}.$$
 We use this representation of $\vf$ to calculate
$$\vc_{k,\ell}:=\langle \vf,\eta_{k,\ell}\rangle,$$
where we obtain $\vc=\mathbf{B}\vs$, with $$\mathbf{B}=(\langle \psi_{j,m},\eta_{k,\ell}\rangle)_{(j,m)(k,\ell)}.$$
 Hence, by Proposition \ref{prop:di},
 $$\|\{\vc_{k,\ell}\}_{k,\ell}\|_{m_{p,q}^{s}(W)}\leq C_2 \|\{\vs_{k,\ell}\}_{k,\ell}\|_{m_{p,q}^{s}(W)}\leq C_1C_2 \|\vf\|_{M_{p,q}^{s}(W)}.$$
\end{proof}

\section{Applications to Fourier multipliers and PDOs}\label{sec:5}
One key consequence of Definition \ref{def:mad} is that the almost-diagonal
class $\textbf{ad}_{{p},q}^{s}(W)$ is compatible with the corresponding
scalar almost-diagonal class, up to a loss dictated by the doubling exponent
of the matrix weight.  Proposition \ref{prop:di} therefore gives a convenient
general way to transfer scalar almost-diagonal estimates to the matrix-weighted
setting.  We illustrate this principle with a Fourier multiplier result and then proceed to obtain an estimate for pseudo-differential operators.

\subsection{Fourier multipliers}
Let $0<p<\infty$ and $0<q<\infty$, and suppose that
$W\in\mathbf{A}_p(\cB_A)$.  For a scalar symbol $m$ we let
$m(D)$ denote the Fourier multiplier acting componentwise on vector-valued
functions,
\[
    m(D)\vf:=\mathcal F^{-1}(m\widehat{\vf}),
\]
initially for $\vf\in \bigoplus_{j=1}^N\mathcal{S}(\bR^d)$.  The following
result treats symbols whose order is measured with respect to the scale
$r_k=h(\xi_k)$ of the covering.

\begin{proposition}\label{prop:fourier_multiplier}
Let $0<p<\infty$, $0<q<\infty$, and suppose that
$W\in\mathbf{A}_p(\cB_A)$.  Fix $b\in\bR$.  Assume that the symbol
$m$ satisfies, for every multi-index $\gamma\in\bN_0^d$,
\begin{equation}\label{eq:h_symbol_order_b}
\sup_{k\in\bZ^d}
 r_k^{-b}
 \left\|
 \partial^\gamma[m\circ T_k]
 \right\|_{L^\infty(B_A(0,\rho))}<\infty,
 \qquad T_k y=\delta_{r_k}y+\xi_k .
\end{equation}
Then $m(D)$ extends uniquely to a bounded operator
\begin{equation}\label{eq:bdd}
    m(D):M_{p,q}^{s+b}(W)\longrightarrow M_{p,q}^{s}(W).
\end{equation}
\end{proposition}

\begin{proof}
Let $\{\psi_{k,\ell}\}$ be the tight frame from \eqref{eq:ltf}.  By
Proposition \ref{prop:coef}, Corollary \ref{cor:recon}, and Proposition
\ref{prop:di}, it suffices to prove that the matrix
\begin{equation}\label{eq:bdda}
\mathbf B=
\left\{
 b_{(j,n)(k,\ell)}
\right\}_{j,n,k,\ell},
\qquad
 b_{(j,n)(k,\ell)}:=
 \left\langle r_k^{-b}m(D)\psi_{k,\ell},\psi_{j,n}\right\rangle,
\end{equation}
belongs to $\textbf{ad}_{p,q}^{s}(W)$.  Indeed, if
$\vc_{k,\ell}=\langle \vf,\psi_{k,\ell}\rangle$, then the coefficient
sequence of $m(D)\vf$ is obtained by applying $\mathbf B$ to the sequence
$\{r_k^b\vc_{k,\ell}\}_{k,\ell}$, whose $m_{p,q}^{s}(W)$ norm is precisely
controlled by the $M_{p,q}^{s+b}(W)$ norm of $\vf$.

The support properties of the frame imply that
$b_{(j,n)(k,\ell)}=0$ whenever $C_k\cap C_j=\emptyset$.  We may therefore
assume that $C_k\cap C_j\neq\emptyset$.  In that case, the moderateness of
the covering gives
\begin{equation}\label{eq:intersecting_scales}
    r_k\asymp r_j,
    \qquad
    1+\max(r_k,r_j)^{-1}|\xi_k-\xi_j|_A\asymp 1,
\end{equation}
with constants independent of $j$ and $k$.

Using Plancherel's identity and the change of variables
$\xi=T_k y=\delta_{r_k}y+\xi_k$, the entry in \eqref{eq:bdda} can be written,
up to a harmless unimodular factor, as
\[
 b_{(j,n)(k,\ell)}
 =c_{k,j}\int_{\bR^d} a_{k,j}(y)
       e^{i\Phi_{j,n,k,\ell}(y)}\,dy,
\]
where
\[
 a_{k,j}(y):=
 r_k^{-b}m(T_k y)\psi_k(T_k y)\psi_j(T_k y),
 \qquad |c_{k,j}|\asymp 1 .
\]
The functions $a_{k,j}$ are supported in a fixed compact set, uniformly for all intersecting pairs $C_k\cap C_j\neq\emptyset$.
Moreover, the symbol hypothesis \eqref{eq:h_symbol_order_b}, the normalized
BAPU estimates, and the Leibniz rule imply that, for every multi-index
$\gamma$,
\begin{equation}\label{eq:akj_uniform}
\sup_{C_k\cap C_j\neq\emptyset}
\|\partial^\gamma a_{k,j}\|_{L^\infty}<\infty .
\end{equation}
Repeated integration by parts in the oscillatory integral therefore gives, for
every $N_0>0$,
\begin{equation}\label{eq:mult_spatial_decay}
 |b_{(j,n)(k,\ell)}|
 \leq C_{N_0}
 \left(1+
 \min(r_k,r_j)|x_{k,\ell}-x_{j,n}|_A
 \right)^{-N_0},
 \qquad C_k\cap C_j\neq\emptyset .
\end{equation}
Here we used the definition $x_{k,\ell}=\delta_{r_k^{-1}}\ell$ from
\eqref{eq:xk} and the equivalence of scales in \eqref{eq:intersecting_scales}.

Let $\beta$ be the doubling exponent of $W$ and choose
$J>\nu/\min\{1,p,q\}$ and $M>\max\{2J,|s|+\nu/2\}$.  Taking
$N_0>J+\beta/p$, the estimate \eqref{eq:mult_spatial_decay}, together with
\eqref{eq:intersecting_scales} and the vanishing for non-intersecting
frequency sets, gives exactly the symmetric estimate \eqref{eq:mad_symm}.
Hence $\mathbf B\in\textbf{ad}_{p,q}^{s}(W)$, and Proposition \ref{prop:di}
yields the boundedness in \eqref{eq:bdd}.
\end{proof}

In the following corollary we use the notation
\[
        \langle D\rangle^b := h^b(D).
\]
For this statement we assume that \(h\) is smooth and that, for every multi-index
\(\gamma\),
\[
\sup_{k\in\mathbb Z^d}
r_k^{-b}
\left\|
\partial^\gamma [h^b\circ T_k]
\right\|_{L^\infty(B_A(0,\rho))}<\infty,
\qquad
\sup_{k\in\mathbb Z^d}
r_k^{b}
\left\|
\partial^\gamma [h^{-b}\circ T_k]
\right\|_{L^\infty(B_A(0,\rho))}<\infty.
\]

\begin{corollary}\label{cor:h_bessel}
Let $0<p<\infty$, $0<q<\infty$, and suppose that
$W\in\mathbf{A}_p(\cB_A)$.  Then
\begin{equation}\label{eq:bessel}
\langle D\rangle^b:M_{p,q}^{s+b}(W)\longrightarrow M_{p,q}^{s}(W)
\end{equation}
is bounded.  Moreover,
\[
 \|\vf\|_{M_{p,q}^{s+b}(W)}
 \asymp
 \|\langle D\rangle^b\vf\|_{M_{p,q}^{s}(W)},
 \qquad \vf\in M_{p,q}^{s+b}(W).
\]
\end{corollary}

\begin{proof}
The boundedness of $\langle D\rangle^b$ follows directly from Proposition
\ref{prop:fourier_multiplier} applied to $m=h^b$.  Applying the same
proposition to $m=h^{-b}$ gives the boundedness of
$\langle D\rangle^{-b}:M_{p,q}^{s}(W)\to M_{p,q}^{s+b}(W)$.  Since
$h^{b}h^{-b}=1$, the two estimates give the stated norm
equivalence.
\end{proof}


\subsection{Pseudo-differential operators}
We conclude by recording a simple consequence of the almost-diagonal
machinery for pseudo-differential operators.  The result should be viewed as
an extension of the corresponding scalar result from \cite[Section~7]{BoruNiel:2008a}
to the present matrix-weighted setting.  We restrict attention to scalar
symbols, acting componentwise on vector-valued functions.

For a scalar symbol $\sigma\in C^\infty(\bR^d\times\bR^d)$ we define
\[
    \sigma(x,D)\vf(x)
    :=(2\pi)^{-d/2}\int_{\bR^d}
        \sigma(x,\xi)\widehat{\vf}(\xi)e^{ix\cdot \xi}\,d\xi,
    \qquad \vf\in \bigoplus_{j=1}^N\dS,
\]
where the operator acts coordinatewise.  We say that $\sigma$ belongs to
$S_h^b$ if, for all multi-indices $\gamma,\eta\in\bN_0^d$,
\begin{equation}\label{eq:pdo_symbol_class}
    \sup_{k\in\bZ^d}
    r_k^{-b}
    \left\|
        \partial_y^\gamma\partial_x^\eta
        \big[\sigma(x,T_k y)\big]
    \right\|_{L^\infty(\bR^d_x\times B_A(0,\rho)_y)}<\infty,
    \qquad T_k y=\delta_{r_k}y+\xi_k .
\end{equation}
The order is therefore measured relative to the scale $r_k=h(\xi_k)$ of the
covering.  We have the following result.

\begin{proposition}\label{prop:pdo_bdd}
Let $0<p<\infty$, $0<q<\infty$, and suppose that
$W\in\mathbf{A}_p(\cB_A)$.  If $\sigma\in S_h^b$, then $\sigma(x,D)$ extends
uniquely to a bounded operator
\begin{equation}\label{eq:pdo_bdd}
    \sigma(x,D):M_{p,q}^{s+b}(W)\longrightarrow M_{p,q}^{s}(W).
\end{equation}
\end{proposition}

\begin{proof}
Let $\{\psi_{k,\ell}\}_{k,\ell}$ be the tight frame from \eqref{eq:ltf}.  As in
the proof of Proposition \ref{prop:fourier_multiplier}, it suffices to prove
that the matrix
\begin{equation}\label{eq:pdo_matrix}
    \mathbf B=
    \left\{b_{(j,n)(k,\ell)}\right\}_{j,n,k,\ell},
    \qquad
    b_{(j,n)(k,\ell)}
    :=\left\langle r_k^{-b}\sigma(x,D)\psi_{k,\ell},\psi_{j,n}\right\rangle,
\end{equation}
belongs to $\textbf{ad}_{p,q}^{s}(W)$.  Indeed, the coefficient sequence of
$\sigma(x,D)\vf$ is obtained by applying $\mathbf B$ to the sequence
$\{r_k^b\vc_{k,\ell}\}_{k,\ell}$, where
$\vc_{k,\ell}=\langle \vf,\psi_{k,\ell}\rangle$.

The scalar estimates from the proof of the pseudo-differential operator
result in \cite[Section~7]{BoruNiel:2008a} apply to the present normalized
symbol condition \eqref{eq:pdo_symbol_class}.  More precisely, the derivative
bounds in the $y$-variable give arbitrary decay away from the spatial diagonal,
while the derivative bounds in the $x$-variable give arbitrary decay away from
the frequency diagonal.  Hence, for every $J>0$ and every $M>0$,
\begin{align}\label{eq:pdo_ad_estimate}
    |b_{(j,n)(k,\ell)}|
    &\leq C_{J,M}
    \min\left\{
        \left(\frac{r_j}{r_k}\right)^M,
        \left(\frac{r_k}{r_j}\right)^M
    \right\}
    \left(1+\min(r_k,r_j)|x_{k,\ell}-x_{j,n}|_A\right)^{-J}
    \nonumber\\
    &\qquad\qquad\times
    \left(1+\max(r_k,r_j)^{-1}|\xi_k-\xi_j|_A\right)^{-J} .
\end{align}
This is exactly the type of scalar almost-diagonal estimate used in Remark
\ref{rem:ad}.

Let $\beta$ be the doubling exponent of $W$ and choose
$J>\nu/\min\{1,p,q\}+\beta/p$ and
$M>K+\max\{2\nu/\min\{1,p,q\},|s|+\nu/2\}$, where
$K=\frac{\beta}{p}$.  Then \eqref{eq:pdo_ad_estimate}, together
with Definition \ref{def:mad}, shows that
$\mathbf B\in\textbf{ad}_{p,q}^{s}(W)$.  Proposition \ref{prop:di} therefore
gives boundedness on $m_{p,q}^{s}(W)$.  Combining this with Proposition
\ref{prop:coef} and Corollary \ref{cor:recon} proves \eqref{eq:pdo_bdd}.
\end{proof}

\appendix
\section{Some technical proofs}\label{sec:complete}
Here we complete the proof of Proposition \ref{prop:complete}.
We recall that $\bigoplus_{j=1}^N \dS$ denotes the family of vector functions $\vf=(f_1,\ldots,f_N)^\top$ with $f_i\in\dS$, $i=1,\ldots,N$. We equip the space with slightly modified seminorms, where we use the notation from Definition \ref{entertainus},
$$p_m(\vf):=\sum_{j=1}^N p_m(f_j),\qquad\text{with } p_m(f_i):=\sup_{\xi\in\bR^d}  h(\xi)^{m/\tau}\sum_{|\eta|\leq m} |\partial^\eta\hat{f}_i(\xi)|.$$
It follows from the growth estimate specified in Eq.\ \eqref{eq:growthh} that the seminorms define the standard topology on $\bigoplus_{j=1}^N \dS$ as considered previously in \eqref{Snorm1}.
As before, we let $\bigotimes_{j=1}^N \mathcal{S}'(\bR^d)$ denote the corresponding family of vector-valued tempered distributions. We need the following lemma.

\begin{lemma}\label{lem:schwartz-sequence-estimate}
The following estimates hold.
\begin{enumerate}
\item[\textup{(i)}]\label{item:schwartz-localized-frequency-estimate}
Let $r\in\bR$ and $m>0$. If $K>ma_+$ and
$N\geq \max\{K,\tau(\nu+r)\}$, then there exists a constant
$C>0$, independent of $k$ and $f$, such that, for every scalar
$f\in\mathcal{S}(\bR^d)$,
\begin{equation}\label{eq:pN}
    \|\brac{\cdot}_A^m\cF^{-1}[\phi_k\hat{f}]\|_{L^\infty}
    \leq C r_k^{-r} p_N(f).
\end{equation}

\item[\textup{(ii)}]\label{item:schwartz-sequence-estimate}
Let $0<p\leq 1$, $s\in\bR$, and let
$\{A_{Q(k,\ell)}\}_{k,\ell}$ be reducing operators of order $p$
associated with $W$. Suppose that this family is strongly doubling
in the sense of the definition preceding Lemma \ref{le:dou}. Then
there exists a Schwartz seminorm $p_M$ such that
\begin{equation}\label{eq:auxiliary-schwartz-estimate}
\sum_k r_k^{\nu(1/p-1/2)-s}\sup_{\ell\in\bZ^d}
\left|A_{Q(k,\ell)}^{-1}
\langle\boldsymbol{\theta},\psi_{k,\ell}\rangle\right|
\leq C p_M(\boldsymbol{\theta})
\end{equation}
for every
$\boldsymbol{\theta}\in\bigoplus_{j=1}^N \mathcal{S}(\bR^d)$.
\end{enumerate}
\end{lemma}
\begin{proof}
We first prove \eqref{eq:pN}. For $K>ma_+$, using the fact that $|\supp(\phi_k)|\asymp r_k^\nu\asymp h(\xi)^\nu$ uniformly for $\xi\in \supp(\phi_k)$, we obtain
\begin{align*}
    \|\brac{\cdot}_A^m\cF^{-1}[\phi_k\hat{f}]\|_{L^\infty}
    &\lesssim \|(1+|x|)^K\cF^{-1}[\phi_k\hat{f}]\|_{L^\infty}\\
    &\lesssim \sum_{\eta\in\bN_0^d:|\eta|\leq K}\|\cF^{-1}\partial^\eta[\phi_k\hat{f}]\|_{L^\infty}\\
    &\lesssim \sum_{\eta\in\bN_0^d:|\eta|\leq K}\|\partial^\eta[\phi_k\hat{f}]\|_{L^1}\\
    &\lesssim \sup_{\xi\in\bR^d} h(\xi)^\nu
        \sum_{\eta\in\bN_0^d:|\eta|\leq K}|\partial^\eta[\phi_k\hat{f}](\xi)|\\
    &\lesssim
        \bigg(\sup_{\xi\in\bR^d} h(\xi)^{\nu+r}
        \sum_{\eta\in\bN_0^d:|\eta|\leq K}|\partial^\eta\hat{f}(\xi)|\bigg)
        \sup_{\xi\in\supp(\phi_k)} h(\xi)^{-r}
        \sum_{\eta\in\bN_0^d:|\eta|\leq K}|\partial^\eta\phi_k(\xi)|\\
    &\lesssim r_k^{-r} p_N(f),
\end{align*}
whenever $N\geq \max\{K,\tau(\nu+r)\}$. This proves \eqref{eq:pN}.

We next prove \eqref{eq:auxiliary-schwartz-estimate}. Fix a reference set $Q_0=Q(k_0,0)$. The strong doubling estimate from Lemma \ref{le:dou}, applied with $Q=Q_0$ and $P=Q(k,\ell)$, observing that $x_{k_0,0}=0$, gives a polynomial growth bound
\begin{equation}\label{eq:inverse-reducing-polynomial-growth}
\|A_{Q(k,\ell)}^{-1}\|
\leq \|A_{Q_0}^{-1}\|\,\|A_{Q_0}A_{Q(k,\ell)}^{-1}\|
\leq C_0 r_k^{M_0}(1+|x_{k,\ell}|_A)^{M_0}
\end{equation}
for a fixed exponent $M_0$ depending only on the doubling parameters and the covering. On the other hand, since $\boldsymbol{\theta}$ is Schwartz and $\{\psi_{k,\ell}\}_{k,\ell}$ is a uniformly localized band-limited system, the usual integration-by-parts estimate gives, for every $L>0$,
\begin{equation}\label{eq:theta-frame-coefficients-rapid-decay}
\left|\langle\boldsymbol{\theta},\psi_{k,\ell}\rangle\right|
\leq C_Lp_L(\boldsymbol{\theta})r_k^{-L}(1+|x_{k,\ell}|_A)^{-L}.
\end{equation}
Combining \eqref{eq:inverse-reducing-polynomial-growth} and \eqref{eq:theta-frame-coefficients-rapid-decay} gives
\[
\sup_{\ell\in\bZ^d}\left|A_{Q(k,\ell)}^{-1}\langle\boldsymbol{\theta},\psi_{k,\ell}\rangle\right|
\leq C_L p_L(\boldsymbol{\theta}) r_k^{M_0-L}
\sup_{\ell\in\bZ^d}(1+|x_{k,\ell}|_A)^{M_0-L}.
\]
Taking $L>M_0$, the supremum on the right-hand side is bounded by $1$. Hence the left-hand side of \eqref{eq:auxiliary-schwartz-estimate} is bounded by
\[
C_Lp_L(\boldsymbol{\theta})
\sum_k r_k^{\nu(1/p-1/2)-s+M_0-L}.
\]
By the construction of the covering, the disjoint subballs from Lemma \ref{braunsugar}, and the growth condition \eqref{eq:growthh}, the sequence $\{r_k\}_k$ has the same summability behavior as the corresponding $h$-scale covering. Thus $\sum_k r_k^{-\nu-\epsilon}<\infty$ for every $\epsilon>0$, after increasing the exponent. Choosing $L$ so large that
\[
\nu(1/p-1/2)-s+M_0-L<-\nu
\]
gives the desired bound. This proves \eqref{eq:auxiliary-schwartz-estimate}.
\end{proof}

\begin{proof}[Completion of the proof of Proposition \ref{prop:complete}]

We consider the proof of claim (i) and first show the embedding $\bigoplus_{j=1}^N\dS\hookrightarrow M^{s}_{p,q}(W)$. Let $\vf\in \bigoplus_{j=1}^N\dS$, and put $w(\cdot):=\|W^{1/p}(\cdot)\|^p$, which, by Lemma \ref{le:sc}, is a scalar $A_{\max\{1,p\}}(\cB_A)$. We will need the fact that a scalar $A_p(\cB_A)$-weight has moderate average growth in the sense that
\begin{equation}\label{eq:gr}
    \int_{\bR^d} w(x) \brac{x}_A^{-(\nu \max\{1,p\}+\epsilon)}\,dx<+\infty,
    \end{equation}
   for any $\epsilon>0$, which follows directly from Lemmas \ref{le:doub} and \ref{le:sq}.

 We have, for $m> \nu\max\{1,p\}$,

\begin{align}
\int_{\bR^d} |W^{1/p}(x)	\phi_k(D)\vf(x)|^p\dx&\leq 
\int_{\bR^d} \|W^{1/p}(x)\|^p|	\phi_k(D)\vf(x)|^p\dx\nonumber\\
&= \int_{\bR^d} w(x)|	\phi_k(D)\vf(x)|^p\dx\nonumber\\
&\leq  \|\brac{\cdot}_A^{m}	\phi_k(D)\vf\|_{L^\infty}^p\int_{\bR^d} w(x)\brac{x}_A^{-m}\dx\nonumber\\
&\leq  C\|\brac{\cdot}_A^{m}	\phi_k(D)\vf\|_{L^\infty}^p,\label{eq:wp}
\end{align}
where we used  \eqref{eq:gr}.  For the scalar function $f_i\in \dS$, we have from Eq.\ \eqref{eq:pN} that for $m>\nu\max\{1,p\}$, $r>0$, and $N>\max\{ma_+,(\nu+r)\tau\} $, 
\begin{equation}\label{eq:thet}
\|\brac{\cdot}_A^m	\phi_k(D)f_i\|_{L^\infty}\leq cr_k^{-r}p_N(f_i),
\end{equation}
with $c$ independent of $f_i$. 
Hence, by \eqref{eq:wp}, we obtain
$$\|	\phi_k(D)\vf\|_{L^p(W)}\leq c'r_k^{-r}p_N(\vf).$$ We may choose $r$ arbitrarily large by increasing $N$, and for sufficiently large $r$, we  obtain
$$\|\vf\|_{M^{s}_{p,q}(W)}=\|\{r_k^s\|	\phi_k(D)\vf\|_{L^p(W)}\}_k\|_{\ell^q}\leq c'p_N(\vf),$$
 for $c'$ independent of $\vf$, where we have used that $\sum_{k} r_k^{-\nu-\epsilon}<\infty$ for $\epsilon>0$, which can be deduced directly  from  the estimate in Eq.\ \eqref{eq:integral}. This provides the desired embedding.
 
 We now turn to the embedding $M^{s}_{p,q}(W)\hookrightarrow \bigoplus_{j=1}^N\mathcal{S}'(\bR^d)$. First, consider the case $1<p<\infty$.
  Take $\vf\in M^{s}_{p,q}(W)$ and let $\boldsymbol{\theta}=(\theta_1,\ldots,\theta_N)^\top\in  \bigoplus_{j=1}^N\dS$. Then, using the smooth resolution of the identity $\sum_k \psi_k^2(\xi)=1$ from Eq.\ \eqref{yesterday}, we obtain
 \begin{align*}
     \langle \vf,\boldsymbol{\theta}\rangle_{\bC^N}
     &=\sum_{k\in\bZ^d}  \langle \psi_k(D)\vf,\psi_k(D)\boldsymbol{\theta}\rangle_{\bC^N}\\
     &=\sum_{k\in\bZ^d}  \langle r_k^{s}W^{1/p}\psi_k(D)\vf,r_k^{-s}W^{-1/p}\psi_k(D)\boldsymbol{\theta}\rangle_{\bC^N}.
 \end{align*}
  This in turn implies the estimate
 $$\int_{\bR^d}|\langle \vf,\boldsymbol{\theta}\rangle_{\bC^N}|\,\dx  \leq \|\{r_k^s\|\psi_k(D)\vf\|_{L^p(W)}\}_k\|_{\ell^\infty}
 \|\{r_k^{-s}\|\psi_k(D)\boldsymbol{\theta}\|_{L^{p'}(W^{-p'/p})}\}_k\|_{\ell^{1}}.$$

It holds true that $W^{-p'/p}\in \mathbf{A}_{p'}$, cf.\ Remark \ref{re:dual}, so we may use the embedding already obtained to conclude that, for $N$ chosen suitably large, $$\|\{r_k^{-s}\|\phi_k(D)\boldsymbol{\theta}\|_{L^{p'}(W^{-p'/p})}\}_k\|_{\ell^{1}}\leq c p_N(\boldsymbol{\theta}).$$

 For $0<q\leq \infty$, we have the embedding 
 $$\|\{r_k^s\|\psi_k(D)\cdot\|_{L^p(W)}\}_k\|_{\ell^\infty} \asymp \|\cdot\|_{{M^s_{p,\infty}(W)}} \leq \|\cdot\|_{{M^s_{p,q}(W)}},$$ which follows directly from the inclusion $\ell^q\subset \ell^\infty$ for $q<\infty$, and is immediate for $q=\infty$. 
This implies that
 $$\int_{\bR^d} |\langle \vf(x),\boldsymbol{\theta}(x)\rangle_{\bC^N}|\dx
 \leq c\|\vf\|_{M^{s}_{p,q}(W)}p_N(\boldsymbol{\theta}),$$
 which proves the continuous embedding $M^{s}_{p,q}(W)\hookrightarrow \bigoplus_{j=1}^N\mathcal{S}'(\bR^d)$.
For $0<p\leq 1$,  we use the frame coefficients directly. Put
$$\vc_{k,\ell}:=\langle \vf,\psi_{k,\ell}\rangle,
\qquad \vd_{k,\ell}:=\langle \boldsymbol{\theta},\psi_{k,\ell}\rangle.$$
We now use that the tight-frame expansion of $\boldsymbol{\theta}$  also converges in $\bigoplus_{j=1}^N\mathcal{S}(\bR^d)$: indeed, the rapid decay of the coefficients $\vd_{k,\ell}=\langle\boldsymbol{\theta},\psi_{k,\ell}\rangle$, together with the uniform Schwartz estimates for the frame elements $\psi_{k,\ell}$, gives unconditional convergence in every Schwartz seminorm; tightness then identifies the limit with $\boldsymbol{\theta}$. This in turn leads to the following identity, using that $\vf\in \bigoplus_{j=1}^N\mathcal{S}'(\bR^d)$,
$$
\langle \vf,\boldsymbol{\theta}\rangle
=\sum_{k,\ell}\langle \vc_{k,\ell},\vd_{k,\ell}\rangle_{\bC^N}.
$$
Let $A_{Q(k,\ell)}$ be a reducing operator of order $p$ associated with $Q(k,\ell)$. Then
$$
|\langle \vc_{k,\ell},\vd_{k,\ell}\rangle_{\bC^N}|
\leq |A_{Q(k,\ell)}\vc_{k,\ell}|\,
|A_{Q(k,\ell)}^{-1}\vd_{k,\ell}|.
$$
Since $0<p\leq 1$, the elementary inequality $\sum_{\ell}a_{\ell}\leq (\sum_{\ell}a_{\ell}^p)^{1/p}$ for non-negative sequences gives
\begin{align*}
|\langle \vf,\boldsymbol{\theta}\rangle|
&\leq \sum_k\left(\sum_{\ell}|A_{Q(k,\ell)}\vc_{k,\ell}|^p\right)^{1/p}
\sup_{\ell}|A_{Q(k,\ell)}^{-1}\vd_{k,\ell}|.
\end{align*}
Using $|Q(k,\ell)|\asymp r_k^{-\nu}$ and Lemma \ref{le:eqi}, this yields
\begin{align*}
|\langle \vf,\boldsymbol{\theta}\rangle|
&\lesssim
\|\{\vc_{k,\ell}\}_{k,\ell}\|_{m^s_{p,\infty}(W)}
\sum_k r_k^{\nu(1/p-1/2)-s}
\sup_{\ell}|A_{Q(k,\ell)}^{-1}\langle \boldsymbol{\theta},\psi_{k,\ell}\rangle|.
\end{align*}
By Proposition \ref{prop:coef}, and the embedding
$M^s_{p,q}(W)\hookrightarrow M^s_{p,\infty}(W)$, we have
$$
\|\{\vc_{k,\ell}\}_{k,\ell}\|_{m^s_{p,\infty}(W)}
\lesssim \|\vf\|_{M^s_{p,\infty}(W)}
\leq \|\vf\|_{M^s_{p,q}(W)}.
$$
The remaining factor is controlled by Lemma \ref{lem:schwartz-sequence-estimate}; hence, for a sufficiently large Schwartz seminorm $p_M$,
$$
|\langle \vf,\boldsymbol{\theta}\rangle|
\leq C\|\vf\|_{M^s_{p,q}(W)}p_M(\boldsymbol{\theta}).
$$
This proves the continuous embedding $M^{s}_{p,q}(W)\hookrightarrow \bigoplus_{j=1}^N\mathcal{S}'(\bR^d)$ in the case $0<p\leq 1$, and completes the proof of part (i).

We turn to the proof of part (ii). Suppose that  $\{\vf_n\}_n$ is a Cauchy sequence in $M^{s}_{p,q}(W)$. The sequence $\{\vf_n\}_n$ is also Cauchy in the complete space $\bigoplus_{j=1}^N\mathcal{S}'(\bR^d)$ by part (i). Hence, the sequence has a well-defined limit $\vf\in \bigoplus_{j=1}^N\mathcal{S}'(\bR^d)$. Recall that we have $\phi_j\in\dS$ so it follows that $\lim_{m\rightarrow \infty}[\phi_j(D)\vf_m](x)= [\phi_j(D)\vf](x)$ pointwise for $x\in\bR^d$.
If $0<q<\infty$, then by an iterated application of Fatou's lemma,
\begin{align*}
    \|\vf-\vf_n\|_{M^{s}_{p,q}(W)}&\asymp
    \bigg(\sum_{j=1}^\infty r_j^{sq}\|\phi_j(D)(\vf-\vf_n)\|_{L^p(W)}^q\bigg)^{1/q}\\
    &\leq \bigg(\sum_{j=1}^\infty r_j^{sq}\liminf_{m\rightarrow \infty}\|\phi_j(D)(\vf_m-\vf_n)\|_{L^p(W)}^q\bigg)^{1/q}\\
    &\leq \liminf_{m\rightarrow \infty}\bigg(\sum_{j=1}^\infty r_j^{sq}\|\phi_j(D)(\vf_m-\vf_n)\|_{L^p(W)}^q\bigg)^{1/q}\\
    &\asymp \liminf_{m\rightarrow \infty} \|\vf_m- \vf_n\|_{M^{s}_{p,q}(W)}<+\infty.
\end{align*}
For $q=\infty$, the same conclusion follows from
\begin{align*}
\|\vf-\vf_n\|_{M^{s}_{p,\infty}(W)}
&\asymp \sup_j r_j^s\|\phi_j(D)(\vf-\vf_n)\|_{L^p(W)}\\
&\leq \liminf_{m\rightarrow\infty}\sup_j r_j^s\|\phi_j(D)(\vf_m-\vf_n)\|_{L^p(W)}\\
&\asymp \liminf_{m\rightarrow\infty}\|\vf_m-\vf_n\|_{M^{s}_{p,\infty}(W)}<+\infty.
\end{align*}
 We first deduce from this estimate that $\vf=(\vf-\vf_n)+\vf_n\in M^{s}_{p,q}(W)$. Moreover,  
it also follows from the same estimate that 
$\vf_n\rightarrow \vf$ in $M^{s}_{p,q}(W)$ by using that
$\{\vf_n\}_n$ is a Cauchy sequence in $M^{s}_{p,q}(W)$. This proves completeness of $M^{s}_{p,q}(W)$ and concludes the proof.
\end{proof}


\begin{thebibliography}{10}
\expandafter\ifx\csname url\endcsname\relax
  \def\url#1{\texttt{#1}}\fi
\expandafter\ifx\csname doi\endcsname\relax
  \def\doi#1{\burlalt{doi:#1}{http://dx.doi.org/#1}}\fi
\expandafter\ifx\csname urlprefix\endcsname\relax\def\urlprefix{URL }\fi
\expandafter\ifx\csname href\endcsname\relax
  \def\href#1#2{#2}\fi
\expandafter\ifx\csname burlalt\endcsname\relax
  \def\burlalt#1#2{\href{#2}{#1}}\fi

\bibitem{BoruNiel:2008a}
L.~Borup and M.~Nielsen.
\newblock On anisotropic {T}riebel-{L}izorkin type spaces, with applications to the study of pseudo-differential operators.
\newblock {\em J. Funct. Spaces Appl.}, 6(2):107--154, 2008.
\newblock \doi{10.1155/2008/510584}.

\bibitem{BuHytoYang:2024a}
F.~Bu, T.~Hyt\"onen, D.~Yang, and W.~Yuan.
\newblock Matrix-weighted {B}esov-type and {T}riebel-{L}izorkin-type spaces {III}: characterizations of molecules and wavelets, trace theorems, and boundedness of pseudo-differential operators and {C}alder\'on-{Z}ygmund operators.
\newblock {\em Math. Z.}, 308(2):Paper No. 32, 67, 2024.
\newblock \doi{10.1007/s00209-024-03584-8}.

\bibitem{BuHytoYang:2025c}
F.~Bu, T.~Hyt\"onen, D.~Yang, and W.~Yuan.
\newblock Matrix-weighted {B}esov-type and {T}riebel-{L}izorkin-type spaces {I}: {$A_p$}-dimensions of matrix weights and {$\psi$}-transform characterizations.
\newblock {\em Math. Ann.}, 391(4):6105--6185, 2025.
\newblock \doi{10.1007/s00208-024-03059-5}.

\bibitem{BuHytoYang:2025b}
F.~Bu, T.~Hyt\"onen, D.~Yang, and W.~Yuan.
\newblock Matrix-weighted {B}esov-type and {T}riebel-{L}izorkin-type spaces {II}: {S}harp boundedness of almost diagonal operators.
\newblock {\em J. Lond. Math. Soc. (2)}, 111(3):Paper No. e70094, 59, 2025.
\newblock \doi{10.1112/jlms.70094}.

\bibitem{Cald:1976a}
A.~Calder{\'o}n.
\newblock Inequalities for the maximal function relative to a metric.
\newblock {\em Studia Mathematica}, 57(3):297--306, 1976.
\newblock \doi{10.4064/sm-57-3-297-306}.

\bibitem{CleaGeor:2020a}
G.~Cleanthous and A.~G. Georgiadis.
\newblock Mixed-norm {$\alpha$}-modulation spaces.
\newblock {\em Trans. Amer. Math. Soc.}, 373(5):3323--3356, 2020.
\newblock \doi{10.1090/tran/8023}.

\bibitem{Feic:1987a}
H.~G. Feichtinger.
\newblock Banach spaces of distributions defined by decomposition methods. {II}.
\newblock {\em Math. Nachr.}, 132:207--237, 1987.
\newblock \doi{10.1002/mana.19871320116}.

\bibitem{FeicGrob:1985a}
H.~G. Feichtinger and P.~Gr\"{o}bner.
\newblock Banach spaces of distributions defined by decomposition methods. {I}.
\newblock {\em Math. Nachr.}, 123:97--120, 1985.
\newblock \doi{10.1002/mana.19851230110}.

\bibitem{FrazJawe:1985a}
M.~Frazier and B.~Jawerth.
\newblock Decomposition of {B}esov spaces.
\newblock {\em Indiana Univ. Math. J.}, 34(4):777--799, 1985.

\bibitem{FrazJawe:1990a}
M.~Frazier and B.~Jawerth.
\newblock A discrete transform and decompositions of distribution spaces.
\newblock {\em J. Funct. Anal.}, 93(1):34--170, 1990.
\newblock \doi{10.1016/0022-1236(90)90137-A}.

\bibitem{FrazRoud:2004a}
M.~Frazier and S.~Roudenko.
\newblock Matrix-weighted {Besov} spaces and conditions of ${A}_p$ type for $0<p\leq 1$.
\newblock {\em Indiana Univ. Math. J.}, 53(5):1225--1254, 2004.
\newblock \doi{10.1512/iumj.2004.53.2483}.

\bibitem{FrazRoud:2021a}
M.~Frazier and S.~Roudenko.
\newblock Littlewood-{P}aley theory for matrix-weighted function spaces.
\newblock {\em Math. Ann.}, 380(1-2):487--537, 2021.
\newblock \doi{10.1007/s00208-020-02088-0}.

\bibitem{GeorJohnNiel:2017a}
A.~G. Georgiadis, J.~Johnsen, and M.~Nielsen.
\newblock Wavelet transforms for homogeneous mixed-norm {T}riebel-{L}izorkin spaces.
\newblock {\em Monatsh. Math.}, 183(4):587--624, 2017.
\newblock \doi{10.1007/s00605-017-1036-z}.

\bibitem{IsraPottTrei:2022a}
J.~Isralowitz, S.~Pott, and S.~Treil.
\newblock Commutators in the two scalar and matrix weighted setting.
\newblock {\em J. Lond. Math. Soc. (2)}, 106(1):1--26, 2022.
\newblock \doi{10.1112/jlms.12560}.

\bibitem{John:2014a}
F.~John.
\newblock Extremum problems with inequalities as subsidiary conditions.
\newblock In G.~Giorgi and T.~H. Kjeldsen, editors, {\em Traces and Emergence of Nonlinear Programming}, pages 197--215. Springer Basel, Basel, 2014.
\newblock \doi{10.1007/978-3-0348-0439-4\_9}.

\bibitem{NazaTrei:1996a}
F.~L. Nazarov and S.~R. Tre\u{\i}l\cprime.
\newblock The hunt for a {B}ellman function: applications to estimates for singular integral operators and to other classical problems of harmonic analysis.
\newblock {\em Algebra i Analiz}, 8(5):32--162, 1996.

\bibitem{Niel:2010a}
M.~Nielsen.
\newblock On stability of finitely generated shift-invariant systems.
\newblock {\em J. Fourier Anal. Appl.}, 16(6):901--920, 2010.
\newblock \doi{10.1007/s00041-009-9096-7}.

\bibitem{mn_geo}
M.~Nielsen.
\newblock Matrix ${A}_p$-weights relative to a pseudo-metric.
\newblock {\em J.\ Geom.\ Anal.}, page (to appear), 2026, \burlalt{arXiv:2510.02849}{http://arxiv.org/abs/arXiv:2510.02849}.

\bibitem{NielRasm:2012a}
M.~Nielsen and K.~N. Rasmussen.
\newblock Compactly supported frames for decomposition spaces.
\newblock {\em J. Fourier Anal. Appl.}, 18(1):87--117, 2012.

\bibitem{NielSiki:2021a}
M.~Nielsen and H.~\v{S}iki\'{c}.
\newblock Muckenhoupt matrix weights.
\newblock {\em J. Geom. Anal.}, 31(9):8850--8865, 2021.
\newblock \doi{10.1007/s12220-020-00440-z}.

\bibitem{NielSiki:2025a}
M.~Nielsen and H.~\v{S}iki\'c.
\newblock Muckenhoupt matrix weights for general bases.
\newblock In {\em The mathematical heritage of {G}uido {W}eiss}, Appl. Numer. Harmon. Anal., pages 361--386. Birkh\"auser/Springer, Cham, 2025.
\newblock \doi{10.1007/978-3-031-76793-7\_17}.

\bibitem{Roud:2003a}
S.~Roudenko.
\newblock Matrix-weighted {B}esov spaces.
\newblock {\em Trans. Amer. Math. Soc.}, 355(1):273--314, 2003.
\newblock \doi{10.1090/S0002-9947-02-03096-9}.

\bibitem{SteiWain:1978a}
E.~M. Stein and S.~Wainger.
\newblock Problems in harmonic analysis related to curvature.
\newblock {\em Bull. Amer. Math. Soc.}, 84(6):1239--1295, 1978.
\newblock \doi{10.1090/S0002-9904-1978-14554-6}.

\bibitem{TreiVolb:1997a}
S.~Treil and A.~Volberg.
\newblock Wavelets and the angle between past and future.
\newblock {\em J. Funct. Anal.}, 143(2):269--308, 1997.
\newblock \doi{10.1006/jfan.1996.2986}.

\bibitem{Trie:1983a}
H.~Triebel.
\newblock Modulation spaces on the {E}uclidean {$n$}-space.
\newblock {\em Z. Anal. Anwendungen}, 2(5):443--457, 1983.
\newblock \doi{10.4171/ZAA/79}.

\bibitem{Volb:1997a}
A.~Volberg.
\newblock Matrix {$A_p$} weights via {$S$}-functions.
\newblock {\em J. Amer. Math. Soc.}, 10(2):445--466, 1997.
\newblock \doi{10.1090/S0894-0347-97-00233-6}.

\end{thebibliography}

\end{document}